\documentclass[12pt, reqno]{amsart}
\usepackage{amsmath, amsthm, amscd, amsfonts, amssymb, graphicx, xcolor}
\usepackage[bookmarksnumbered, colorlinks, plainpages]{hyperref}
\usepackage{comment}
\usepackage{dsfont}
\usepackage{tikz}
\usetikzlibrary{shapes.misc} 
\tikzset{cross/.style={cross out, draw=red, minimum size=2*(#1-\pgflinewidth), inner sep=0pt, outer sep=0pt},
cross/.default={1pt}}
\usetikzlibrary{automata,positioning}
\usepackage{subcaption}
\usepackage{listings}
\usepackage{multirow}
\usepackage{multicol}
\usepackage{grffile} 
\graphicspath{{./Images/}}

\allowdisplaybreaks

\lstdefinestyle{mystyle}{
    backgroundcolor=\color{backcolour},   
    commentstyle=\color{codegreen},
    keywordstyle=\color{magenta},
    numberstyle=\tiny\color{codegray},
    stringstyle=\color{codepurple},
    basicstyle=\ttfamily\footnotesize,
    breakatwhitespace=false,         
    breaklines=true,                 
    captionpos=b,                    
    keepspaces=true,                 
    numbers=left,                    
    numbersep=5pt,                  
    showspaces=false,                
    showstringspaces=false,
    showtabs=false,                  
    tabsize=2
}

\newtheorem{theorem}{Theorem}[section]
\newtheorem{lemma}[theorem]{Lemma}
\newtheorem{proposition}[theorem]{Proposition}
\newtheorem{corollary}[theorem]{Corollary}
\theoremstyle{definition}

\theoremstyle{remark}

\numberwithin{equation}{section}

\newtheorem*{assumption}{Assumption}

\renewcommand{\assumption}[2]{\textit{Assumption}\hspace{1pt}#1: #2}

\RequirePackage[backend=biber,style=apa]{biblatex}
\newcommand{\citelink}[2]{\hyperlink{cite.\therefsection @#1}{\textcolor{blue}{#2}}}
\newcommand{\dd}{\mathop{}\mathopen{}\mathrm{d}}

\newcommand{\cercle}{\mathbb{S}^1}
\newcommand{\Inte}{\int_{\cercle}}
\newcommand{\argmin}[1]{\underset{#1}{\mathrm{argmin}}}
\newcommand{\indices}{\lambda\in \Lambda_m}
\newcommand{\philambda}{\varphi_{\lambda}}
\newcommand{\Proba}{\mathbb P}
\newcommand{\Esp}{\mathbb E}
\DeclareMathOperator{\Var}{Var}

\makeatletter

\renewcommand{\tocsection}[3]{%
  \indentlabel{\@ifnotempty{#2}{\bfseries\ignorespaces#1 #2\quad}}\bfseries#3}
\renewcommand{\tocsubsection}[3]{%
  \indentlabel{%
    \hbox to 3em{#2\hfil}%
  }#3%
}

\newcommand\@dotsep{4.5}
\def\@tocline#1#2#3#4#5#6#7{\relax
  \ifnum #1>\c@tocdepth 
  \else
    \par \addpenalty\@secpenalty\addvspace{#2}%
    \begingroup \hyphenpenalty\@M
    \@ifempty{#4}{%
      \@tempdima\csname r@tocindent\number#1\endcsname\relax
    }{%
      \@tempdima#4\relax
    }%
    \parindent\z@ \leftskip#3\relax \advance\leftskip\@tempdima\relax
    \rightskip\@pnumwidth plus1em \parfillskip-\@pnumwidth
    #5\leavevmode\hskip-\@tempdima{#6}\nobreak
    \leaders\hbox{$\m@th\mkern \@dotsep mu\hbox{.}\mkern \@dotsep mu$}\hfill
    \nobreak
    \hbox to\@pnumwidth{\@tocpagenum{\ifnum#1=1\bfseries\fi#7}}\par
    \nobreak
    \endgroup
  \fi}
\AtBeginDocument{%
\expandafter\renewcommand\csname r@tocindent0\endcsname{0pt}
}
\def\l@subsection{\@tocline{2}{0pt}{2.5pc}{5pc}{}}

\makeatother

\begin{document}

\setcounter{page}{1}

\centerline{}

\centerline{}


\title[Nonparametric estimation for censored circular data]{Nonparametric estimation for censored circular data}

\author[N. Conanec]{Nicolas Conanec}

\address{ LAGA (UMR 7539), Universit\'e Sorbonne Paris Nord, Institut Galil\'ee, 99 avenue Jean-Baptiste Clement, 93430 Villetaneuse, France.}
\email{\textcolor[rgb]{0.00,0.00,0.84}{conanec@math.univ-paris13.fr}}


\begin{abstract}
We study the problem of estimating the probability density function of a circular random variable subject to censoring. To this end, we propose a fully computable quotient estimator that combines a projection estimator on linear sieves with a method-of-moments approach. We derive an upper bound for its mean integrated squared error and establish convergence rates when the underlying density lies in a Sobolev class. The practical performance of the estimator is illustrated through simulated examples.
Nonparametric estimator, Penalized contrast
\newline
\end{abstract} \maketitle

\tableofcontents

\section{Introduction}\setcounter{equation}{0}

\vphantom{ \cite{Matthews} \cite{Curray} \cite{Anderson-Wu} \cite{MardiaJupp} \cite{Jammalamadaka-Book} \cite{Ley-Verdebout} \cite{Kaplan-Meier} \cite{Turnbull-1974} \cite{Jammalamadaka-2009} \cite{Jammalamadaka-2003} \cite{Efron} \cite{Brunel-Comte-2009} \cite{Talagrand} \cite{Klein-Rio} \cite{Birge-Massart-1998} \cite{Diamond-McDonald-1992} \cite{Wellner-1995} \cite{vanDerLaan-Jewell-1995} \cite{Bhattacharyya-1985} \cite{Sirvanci-Yang-1984} \cite{Parzezn-1962} \cite{Rosenblatt-1956} \cite{Cencov-1962} \cite{Stone-1977} \cite{Kiefer-Wolfowitz-1956} \cite{Efromovich-2021} \cite{Alotaibi-2025} \cite{Tsybakov} \cite{DiMarzio-2012} \cite{Chaubey-2022} \cite{Baudry-2012} }

Directional statistics is the branch of statistics that deals with data represented as directions. Such data arise in a variety of contexts, particularly when measurements involve instruments like gyroscopes for three-dimensional directions or clocks for time. Studying directional data is relevant to understanding natural phenomena such as animal migration (see \citelink{Matthews}{Matthews ('61)}), wind direction (see \citelink{Curray}{Curray ('56)}), daily event occurrence patterns, or automotive flywheel rotation (see \citelink{Anderson-Wu}{Anderson \& Wu ('95)}). In this paper, we focus on circular data, meaning that the phenomenon of interest, represented by the random variable $X$, takes values on the unit circle $\mathbb{S}^1$. Since such data behave differently from data on the real line, specific statistical tools and methodologies are required. The field of directional statistics has been thoroughly studied in several monographs, including \citelink{MardiaJupp}{Mardia \& Jupp (2000)}, \citelink{Jammalamadaka-Book}{Jammalamadaka \& Sengupta (2001)}, and more recently \citelink{Ley-Verdebout}{Ley \& Verdebout (2017)}.

In many applications, full observations are not available due to censoring, where only partial information about the variable of interest is observed. Various censoring models exist, each providing different levels of information. Censored data problems have been widely studied on the real line. For instance, the well-known Kaplan-Meier estimator \citelink{Kaplan-Meier}{Kaplan \& Meier ('58)} provides a nonparametric maximum likelihood estimator (NPMLE) for survival functions for censored data. This estimator has been extended in several directions: \citelink{Turnbull-1974}{Turnbull ('74)} considered doubly censored failure time data, \citelink{Wellner-1995}{Wellner ('95)} studied interval censoring case 2, and \citelink{Diamond-McDonald-1992}{Diamond \& McDonald ('92)} investigated the current status data model. Recently, \citelink{Efromovich-2021}{Efromovich (2021)} proposed a sharp minimax adaptive estimator for the cumulative distribution function in the case of missing data for the current status model, while \citelink{vanDerLaan-Jewell-1995}{van der Laan \& Jewell ('95)} generalized the model to the doubly censored current status setting. In most settings, when we observe censored data we observe the entire sample of data, known as Type I censoring. Alternatively, one may assume that only a fixed number of censored observations are available; this is referred to as Type II censoring (see \citelink{Bhattacharyya-1985}{Bhattacharyya ('85)} or more recently \citelink{Alotaibi-2025}{Alotaibi et al. (2025)}).

When it comes to censored estimation problems on the circle, most censoring models from the real line cannot be transposed in a well-posed way. To the best of our knowledge, the only well-defined model of censored data on the circle is the one introduced by \citelink{Jammalamadaka-2009}{Jammalamadaka \& Mangalam (2009)} in the context of cumulative distribution function estimation. The aim of this work is to fill the gap in the literature on nonparametric estimation for circular data under censoring.

Nonparametric density estimation is a well-studied problem on the real line; see, for example, \citelink{Tsybakov}{Tsybakov (2009)} for a review of classical techniques. On $\mathbb{S}^1$, most contributions to density estimation are parametric, and only a few are nonparametric. Among these, kernel density estimation has been the most investigated approach (see \citelink{DiMarzio-2012}{DiMarzio et al. (2012)} or \citelink{Chaubey-2022}{Chaubey (2022)}).

In this work, we study the estimation of the probability density function of a circular distribution using interval-censored observations. We rely on the model introduced by \citelink{Jammalamadaka-2009}{Jammalamadaka \& Mangalam (2009)}, which is one of the few censorship models that are well-defined on the circle and can be interpreted as an extension of the doubly censored failure time model from the real line to $\mathbb{S}^1$. In their work, Jammalamadaka and Mangalam estimated the cumulative distribution function using a self-consistent estimator, which coincides with the NPMLE. In contrast, we aim to estimate the density function $f$ using a nonparametric method. We construct an explicit quotient estimator $\hat{f} = \frac{(\hat{\psi}_m)_+}{\hat{\sigma}\vee n^{-1/2}}$, where $\hat{\sigma}$ is an estimator of a function $\sigma$ related to the censoring mechanism, and $\hat{\psi}_m$ is an estimator of the function $\psi = f \sigma$. We define $\hat{\psi}_m$ as a projection estimator onto a linear sieve space $S_m$, and estimate $\sigma$ using a method of moments. $n^{-1/2}$ is a threshold used to make sure the quotient is always well-defined. To quantify the performance of our estimator, we use the mean integrated squared error (MISE). We first derive an upper bound for the MISE of $\hat{f}$, which leads to a rate of convergence over Sobolev classes when the smoothness of $\psi$ is known. Since the smoothness of $\psi$ is typically unknown in practice, we propose a data-driven procedure to select the optimal projection level $\hat{m}$ via a penalized criterion. We show that the penalized estimator $\hat{\psi}_{\hat{m}}$ satisfies an oracle inequality, which means that our selection procedure achieves the best possible trade-off between bias and variance. Furthermore, $\hat{\psi}_{\hat{m}}$ retains the same rate of convergence as the previous estimator when estimating univariate Sobolev functions. In this sense, the estimator $\hat{\psi}_{\hat{m}}$ is adaptive, as it does not require prior knowledge of the smoothness of $\psi$. Returning to the estimation of $f$, we define $\hat{f}^* = \frac{(\hat{\psi}_{\hat{m}})_+}{\hat{\sigma}\vee n^{-1/2}}$ and prove that it also satisfies an oracle inequality. Finally, we show that $\hat{f}^*$ can be efficiently computed in practice. We evaluate its numerical performance under different degrees of censoring and compare our results with the simulations from \citelink{Jammalamadaka-2009}{Jammalamadaka \& Mangalam (2009)}. Overall, our estimator performs very satisfactorily.

This article is structured as follows. In Section~\ref{SectionModele}, we discuss specific features of the circular nature of our data, the effect of censoring on the observed information, and the construction of our estimator. Section~\ref{SectionResultats} presents the theoretical results regarding the proposed estimation procedure. In Section~\ref{SectionSimulations}, we provide simulation studies to evaluate the performance of our estimator under various scenarios. Finally, Section~\ref{SectionPreuves} contains the proofs of the main results.

\section{Model and estimators}\setcounter{equation}{0}\label{SectionModele}

\subsection{Circular context}

First, we need to address issues related to the circular nature of our data. In particular, since we are working with an interval censoring model, a natural question arises: how should one define whether a point belongs to an interval on the circle? More generally, can the usual partial order on $\mathbb{R}$ be extended to $\mathbb{S}^1$? Without further assumptions, it is impossible to say, for instance, whether the angle $\frac{\pi}{2}$ is larger than $\frac{3\pi}{2}$ in any meaningful sense. To resolve this, we adopt the following convention: we fix an initial direction, denoted as $0$, and choose an orientation on the circle. Each angle is then represented by its unique equivalent in $[0, 2\pi[$ modulo $2\pi$. Under these conventions, the circle $\cercle$ can be represented as the interval $[0, 2\pi[$, and every angle as a point on the interval. We then use the usual partial order on $[0, 2\pi[$ as our circular order. However, another issue remains: how do we define an interval between two points on $\mathbb{S}^1$? On the real line, two points define a unique interval between them. On the circle, this is no longer the case. Given two points, and depending on the orientation, there are two different arcs connecting them, corresponding to two disjoint sets whose union is the entire circle. These two sets are complementary. Therefore, from now on, whenever we write an interval $[x, y]$ on the circle, we will mean the set of points going from $x$ to $y$ in the chosen orientation. In other words, for any $(x, y)\in\left(\mathbb{S}^1\right)^2, \mathbb{S}^1 = [x, y[\sqcup [y, x[$.

Thus if we have $x > y$ two points of the interval, $[x, y]$ is the set $[0, 2\pi[ \backslash ]y, x[$, i.e $[x, y] = [x, 2\pi[\sqcup [0, y]$, and $]y, x[$ is the classic interval as we know it.\\
We choose for the rest of the article the anticlockwise orientation. In Figure~\ref{fig:intervalle_circulaire_cercle_et_r} we can see two circles with intervals $\left[ \frac{\pi}{3}, \frac{\pi}{6} \right]$ and $\left[\frac{\pi}{6}, \frac{\pi}{3} \right]$ highlighted and the same intervals on the real line.\\
Moreover on the real line we know thanks to the property of the partial order that $\mathds{1}_{\{x\in[L, U]\}} = \mathds{1}_{\{L\leq x\}} - \mathds{1}_{\{U < x\}}$. This equality is false on $\mathbb{S}^1$ since we need to know what interval we are considering. Lemma~\ref{EqCirculaire} (proof in Section~\ref{Circulaire}) tells us that, if $(x, L, U)$ is a triplet of points in $\cercle$ with $L\ne U$, then we have
\begin{align}
\label{eq:Circulaire}
\mathds{1}_{\{x\in[L, U]\}}&= \mathds{1}_{\{L\leq x\}} - \mathds{1}_{\{U < x\}} + \mathds{1}_{\{L\geq U\}}.
\end{align}

\subsection{The circular censor model}
Set $(\Omega, \mathcal{F},\Proba)$ a probability space. We work with functions in $\mathbb{L}^2(\cercle)$, endowed with the usual scalar product  \\$<g,h> = \Inte g(x) h(x) \dd x$ and the associated norm will be written $\|\cdot\|_2$.
Let $(X_1, L_1, U_1), \dots , (X_n, L_n, U_n)$ be a sample of the triplet $(X, L, U)$ where $(L, U)$ are the censor elements and $X$ is the variable of interest, such that the sequences $(L_i, U_i)_{1\leq i \leq n}$ and $(X_i)_{1\leq i \leq n}$ are independent. The way the censorship works is the following. We exactly observe the data $X_i$ if $X_i\in [L_i, U_i]$, thus we say that $[L_i, U_i]$ is the window of observation, otherwise we do not observe $X_i$ thus we only have the information of the couple $(L_i, U_i)$. In that case we will set the observation to an arbitrary point outside of $[0, 2\pi]$, we choose $-\pi$. We suppose that 
\[L \ne U \textit{a.s.} ,\]
 meaning that for each observation the window of observation is a non null set. Notice here that we do not make the hypothesis $L_i \leq U_i$ since the order is important for the definition of the interval $[L_i, U_i]$, their role is not exchangeable.
\begin{figure}[ht]

\begin{center}
\begin{tikzpicture}[scale=1.5]

  \begin{scope}[xshift=-2cm, yshift=1.6cm]
    \draw[red, thick] (0,0) circle(1);
    
    \foreach \angle/\label in {
        0/{$0$},
        90/{$\frac{\pi}{2}$},
        180/{$\pi$},
        270/{$\frac{3\pi}{2}$}
    } {
      \draw (\angle:1.2) node {\label};
      \draw (-1.05,0) -- (-0.95,0);
      \draw (1.05,0) -- (0.95,0);
      \draw (0,-1.05) -- (0,-0.95);
      \draw (0,1.05) -- (0,0.95);
    }

    \node at (0,0) {+};

    \draw[black, thick] (20:1) arc (20:60:1);
    \fill[red] (20:1) circle(0.025);
    \fill[red] (60:1) circle(0.025);
    
    \foreach \a in {70, 80,...,350} {
      \draw[red, thick] (\a:1) node[cross=2.1pt, rotate=\a] {};
    }
    \foreach \a in {0, 10, 20} {
      \draw[red, thick] (\a:1) node[cross=2.1pt, rotate=\a] {};
    }
  \end{scope}

  \begin{scope}[xshift=-2cm, yshift=-1.6cm]
    \draw[black, thick] (0,0) circle(1);
    
    \foreach \angle/\label in {
        0/{$0$},
        90/{$\frac{\pi}{2}$},
        180/{$\pi$},
        270/{$\frac{3\pi}{2}$}
    } {
      \draw (\angle:1.2) node {\label};
      \draw (-1.05,0) -- (-0.95,0);
      \draw (1.05,0) -- (0.95,0);
      \draw (0,-1.05) -- (0,-0.95);
      \draw (0,1.05) -- (0,0.95);
    }

    \node at (0,0) {+};

    \draw[red, thick] (20:1) arc (20:60:1);
    \fill[red] (20:1) circle(0.025);
    \fill[red] (60:1) circle(0.025);
    
    \foreach \a in {30, 40, 50} {
      \draw[red, thick] (\a:1) node[cross=2.1pt, rotate=\a] {};
    }
  \end{scope}
  
  \begin{scope}[xshift=2cm, yshift=1.6cm]
  
  \draw[red, thick] (0,0) -- (3.14,0);
  
   \foreach \angle/\label in {
        {0*3.14/360}/{$0$},
        {90*3.14/360}/{$\frac{\pi}{2}$},
        {180*3.14/360}/{$\pi$},
        {270*3.14/360}/{$\frac{3\pi}{2}$},
        {360*3.14/360}/{$2\pi$}
    } {
      \draw (\angle,0.5) node {\label};
      \draw (\angle,0.05) -- (\angle,-0.05);
    }
    
    
    \draw[black, thick] (20*3.14/360,0) -- (60*3.14/360,0);
    \fill[red] (20*3.14/360,0) circle(0.025);
    \fill[red] (60*3.14/360,0) circle(0.025);
    
    \foreach \a in {70,80,...,360} {
    \pgfmathsetmacro{\x}{\a*3.14/360}
      \draw[red, thick] (\x,0) node[cross=1.7pt] {};
    }
    \foreach \a in {0, 10} {
    \pgfmathsetmacro{\x}{\a*3.14/360}
      \draw[red, thick] (\x,0) node[cross=1.7pt] {};
    }
    
  \end{scope}
  
   \begin{scope}[xshift=2cm, yshift=-1.6cm]
  
  \draw[black, thick] (0,0) -- (3.14,0);
  
   \foreach \angle/\label in {
        {0*3.14/360}/{$0$},
        {90*3.14/360}/{$\frac{\pi}{2}$},
        {180*3.14/360}/{$\pi$},
        {270*3.14/360}/{$\frac{3\pi}{2}$},
        {360*3.14/360}/{$2\pi$}
    } {
      \draw (\angle,0.5) node {\label};
      \draw (\angle,0.05) -- (\angle,-0.05);
    }
    
    
    \draw[red, thick] (20*3.14/360,0) -- (60*3.14/360,0);
    \fill[red] (30*3.14/360,0) circle(0.025);
    \fill[red] (60*3.14/360,0) circle(0.025);
    
    \foreach \a in {30, 40, 50} {
    \pgfmathsetmacro{\x}{\a*3.14/360}
      \draw[red, thick] (\x,0) node[cross=1.7pt] {};
    }
    
  \end{scope}

\end{tikzpicture}
\end{center}
\caption{On the top left is the circular interval $\left[ \frac{\pi}{3}, \frac{\pi}{6} \right]$ on $\cercle$ and on the top right on  $\mathbb{R}$. On the bottom left is the circular interval $\left[ \frac{\pi}{6}, \frac{\pi}{3} \right]$ on $\cercle$ and on the bottom right on $\mathbb{R}$.}
  \label{fig:intervalle_circulaire_cercle_et_r}
\end{figure}
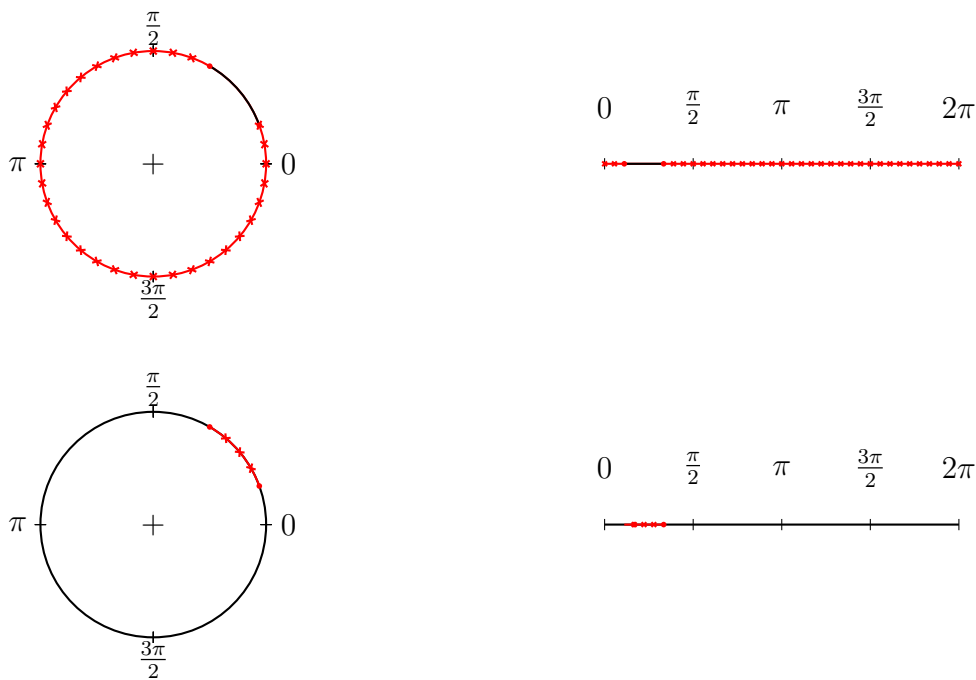\\
We then define
\begin{equation}\nonumber
\left\{
\begin{array}{ll}
X_i'=\left\{
\begin{array}{ll}
        X_i \text{ , if } X_i \in [L_i, U_i] ,\\
        -\pi \text{ , otherwise.}
    \end{array}
\right.\\
\Delta_i= \mathds{1}_{\{X_i\in[L_i, U_i]\}}.
\end{array}
\right.
\end{equation}
So the true observations are the sample of triplets $(X_1', L_1, U_1),\dots, (X_n', L_n, U_n)$.\\

The goal is to estimate the density function $f$ of the circular variable $X$, using a sample of $(X'\hspace{-2pt}, L, U)$. For that purpose we use a quotient estimator.  Consider the following function:
\begin{equation}\label{eq:Phi}
\sigma: x\in\mathbb{S}^1\mapsto \Proba (x\in [L, U]) =\Inte\Inte \mathds{1}_{\{x\in[l, u]\}} \Proba_{(L, U)}(\dd l, \dd u).
\end{equation}
We can show that, for any integrable function $t$, 

\begin{align}
\nonumber
\Esp(\Delta_i t(X_i'))&= \Esp(\Delta_i t(X_i))\\ 
\nonumber
&= \Inte\Inte\Inte \mathds{1}_{\{x\in[l, u]\}} t(x) f(x) \Proba_{(L, U)}(\dd l, \dd u) \dd x \\
\label{eq:Esperance_et_psi}
&= \Inte t(x) f(x)\sigma(x) \dd x= <t, f\sigma> = <t, \psi>,
\end{align}
where we set 
\[ \psi = f\sigma .\]\\
This equation provides a method to build our estimator of $f$. Indeed since $f = \frac{\psi}{\sigma}$ we hope that an estimation of $\psi$ and $\sigma$ will give us an estimation of the density, when this estimation is well defined.

\subsection{Our estimator}\label{2.3}

For the estimation of $\psi$ we use a minimum contrast estimator. So before defining the contrast function we need to define the space we want the estimator to be an element of. The spaces we consider here are the linear sieves. We recall their definition. The linear sieves are a collection $(S_m)_{m\in \mathcal{M}_n}$ of finite dimensional sub-space of $\mathbb{L}^2(\cercle)$ such that $\dim S_m = D_m$ with $D_m\leq n, \forall m\in \mathcal{M}_n$ and the following inequality is verified:
\begin{equation}\label{eq:Linear_Sieves}
\exists \Phi_0 >0, \forall m \in \mathcal{M}_n, \forall t\in S_m, \|t\|_{\infty} \leq \Phi_0 \sqrt{D_m} \|t\|_2.
\end{equation}
Moreover, for all $m\in\mathcal{M}_n$, if $(\philambda)_{\indices}$ is a orthonormal basis of $S_m$, where $|\Lambda_m|=D_m$ , an equivalent version of \eqref{eq:Linear_Sieves} is
\begin{equation}\label{eq:Linear_Sieves2}
\exists \Phi_0>0, \forall m \in \mathcal{M}_n, \|\sum_{\indices} \philambda^2\|_{\infty} \leq \Phi_0^2 D_m.
\end{equation}
Here $\mathcal{M}_n$ is the set of possible values of $m$.\\
We recall some examples of useful linear sieves$:$
\begin{itemize}
\item[$\bullet$] Regular polynomial spaces$:$ $S_m$ is generated by $m(r+1)$ polynomials of degrees from $0$ to $r$ on each subintervals $\left[\frac{2j\pi}{m}, \frac{2(j+1)\pi}{m} \right]$ for $j \in \{ 0, \dots , m-1 \}$. Thus we have $D_m= m(r+1)$ and we can consider $\mathcal{M}_n = \{ 1, \dots , \lfloor n/(r+1)\rfloor \}$\\
\item[$\bullet$] Trigonometric spaces$:$ $S_m$ is generated by $\{\varphi_0=\frac{1}{\sqrt{2\pi}}, \varphi_{2j-1}=\frac{1}{\sqrt{\pi}}\cos(j \cdot), \\ \varphi_{2j}=\frac{1}{\sqrt{\pi}}\sin(j \cdot) | \text{ for } j \in \{ 1, \dots , m \} \}$, thus $D_m=2m+1$, $\Phi_0 =\frac{1}{\sqrt{2\pi}}$ and we can consider $\mathcal{M}_n = \{ 1, \dots ,  [n/2]-1\}$.\\
\item[$\bullet$] Dyadic wavelet spaces$:$ $S_m$ is generated by $\{ \phi_{j_0,k}, \Psi_{j,k} , k\in\mathbb{Z} , m \geq j \geq j_0\}$ for any fixed resolution $j_0$ and with $\phi_{j_0,k}(x) = \sqrt{\frac{2^{j_0}}{2\pi}}\phi\left(\frac{2^{j_0}x}{2\pi} - k\right)$ and $\Psi_{j,k}(x) = \sqrt{\frac{2^{j}}{2\pi}}\Psi\left(\frac{2^j x}{2\pi} - k\right)$ where $\phi$ and $\Psi$ are respectively the scaling function and the mother wavelet on $\cercle$ and are elements of the Hölder space $C^r$ for $r\geq 0$.\\
\end{itemize}

Motivated by \eqref{eq:Esperance_et_psi} we define the following empirical contrast function:
\begin{equation}\label{eq:Contraste}
\gamma_n : t\in \mathbb{L}^2(\cercle) \mapsto \|t\|_2^2 - \frac{2}{n}\sum_{i=1}^n \Delta_i t(X_i').
\end{equation}
Thus we define the estimator of $\psi$ as the following. For $m$ in $\mathcal{M}_n$,
\[ \hat{\psi}_m = \argmin{t\in S_m}\hspace{3pt} \gamma_n(t) = \sum_{\indices} \hat{a}_\lambda \philambda, \]
where the $(\hat{a}_\lambda)_{\indices}$ will be specified later. The reason we define the contrast like this is the following. If we take its expectation using \eqref{eq:Esperance_et_psi} we have, for $t$ in $\mathbb{L}^2(\cercle)$,
\[ \Esp( \gamma_n(t)) = \|t\|_2^2 - 2<t,\psi> = \|t-\psi\|_2^2 - \|\psi\|^2_2. \]
The function $\hat{\psi}_m$ which minimizes the contrast $\gamma_n$ on $S_m$ is likely to minimize the norm $\|t-\psi\|_2^2$ on $S_m$, thus to be a relevant estimator of $\psi$.\\
We consider $\psi_m$ the orthogonal projection of $\psi$ on $S_m$. Thus if $(\philambda)_{\indices}$ is an orthonormal basis of $S_m$ we can write
\begin{equation}\label{eq:Def_psi_m}
\psi_m = \sum_{\indices} a_{\lambda}\philambda,
\end{equation}
with 
\begin{equation}\label{eq:def_de_a}
a_{\lambda} = <\psi , \philambda> = \Inte \psi(x) \philambda(x) \dd x = \Esp(\Delta\philambda(X')).
\end{equation} 
This means that the $(a_\lambda)_{\indices}$ are exactly the Fourier coefficients of $\psi_m$ and \eqref{eq:def_de_a} shows that $\tilde{a}_{\lambda} = \frac{1}{n}\sum_{i=1}^n \Delta_i \philambda (X_i')$ is a good estimation of $a_\lambda$, thus making $\sum_{\indices} \tilde{a}_{\lambda}\philambda$ a good estimation of $\psi_m$.\\
The next lemma shows that the coefficients of $\hat{\psi}_m$, the $(\hat{a}_\lambda)_{\indices}$, are exactly the empirical estimator of the Fourier coefficients.
\begin{lemma}\label{lemme_sur_les_coefs}
For all $\indices$, $\hat{a}_\lambda = \tilde{a}_\lambda = \frac{1}{n}\sum_{i=1}^n \Delta_i \philambda(X'_i)$.
\end{lemma}
The proof can be found in Section~\ref{Preuve_des_coefs}.\\
Thus the estimator of $\psi$ can be written as
\begin{equation}\label{eq:Def_hat_psi_m}
\hat{\psi}_m = \sum_{\indices}\left( \frac{1}{n} \sum_{i=1}^n \Delta_i \philambda(X_i')\right)\philambda.
\end{equation}
For the estimation of $\sigma$ defined in \eqref{eq:Phi}, since it is a probability we consider the empirical estimator
\begin{equation}\label{eq:def_har_sigma}
\hat{\sigma} : x\in\cercle\mapsto \frac{1}{n} \sum_{i=1}^n \mathds{1}_{\{x\in[L_i, U_i]\}}.
\end{equation}
Because $f=\frac{\psi}{\sigma}$, we put a threshold to make sure it is always defined. Moreover we make it positive as it is the estimator of a density probability function. We then define the estimator of $f$ as the following:

\begin{equation}\label{Deuxieme estimateur}
\hat{f}: x\mapsto \frac{(\hat{\psi}_m(x))_+}{\hat{\sigma}(x) \lor n^{-1/2}} ,
\end{equation}
where $a \vee b = \max(a,b)$ and $(h(x))_+ =\max(h(x), 0) $ for a function $h$. \\
Let us now prove theoretical results about the estimator $\hat{f}$ and see what hypotheses are necessary to ensure its performances.

\section{Theoretical results}\setcounter{equation}{0}\label{SectionResultats}

\subsection{MISE upper bounds}

We make the following assumption on the model:\\[6pt]
\assumption{(A)}{
There exists a real $\sigma_0>0$ such that, for all $x$ in $\cercle$,\[ \sigma(x)\geq \sigma_0 >0.\]}

This means that, with non-zero probability, any point of the circle can be inside a window of observation. 
This assumption is a theoretical guarantee that any point $x$ of $\cercle$ can be observed and if we have a sample large enough it will be estimated by $\frac{(\hat{\psi}_m (x))_+}{\hat{\sigma}(x)}$. Even though $\sigma_0$ is an unknown quantity it will still be useful for theoretical computations. In some practical cases this assumption will not be verified (for example with deterministic censorship elements) and we will treat them in our simulations. The use of the threshold $n^{-1/2}$ is also here to prevent the cases where the sample is not large enough or $\sigma_0$ is too small resulting in $\hat{\sigma}(x)=0$ which makes our estimator ill-defined.\\
We define $\tilde{f}: x\mapsto \frac{\hat{\psi}_m(x)}{\hat{\sigma}(x) \lor n^{-1/2}}$ such that $\hat{f} = \tilde{f}_+$. We can show that for all $x\in\cercle$, $|\hat{f}(x)-f(x)| \leq |\tilde{f}(x) - f(x)|$ (see Section~\ref{Preuve MISE estimateur seuil}). So the following computations will use $\tilde{f}$.\\
First we show that
\begin{align*}
| \tilde{f} - f | &= \left| \frac{\hat{\psi}_m}{\hat{\sigma}\lor n^{-1/2}} - \frac{\psi}{\sigma} \right| \leq \left| \frac{\hat{\psi}_m - \psi}{\hat{\sigma}\lor n^{-1/2}} \right| + \left| \frac{\sigma - \hat{\sigma}\lor n^{-1/2}}{\hat{\sigma}\lor n^{-1/2}} f\right|.
\end{align*}
Meaning that
\[ \|\tilde{f} - f\|_2^2 \leq 2\underbrace{\left\| \frac{\hat{\psi}_m - \psi}{\hat{\sigma}\lor n^{-1/2}}\right\|_2^2}_{:=A_1} + 2\underbrace{\left\| \frac{\sigma - \hat{\sigma}\lor n^{-1/2}}{\hat{\sigma}\lor n^{-1/2}} f\right\|_2^2}_{:=A_2} .\]
Then we define the following random set $E = \left\{\omega\in\Omega, \forall x\in \cercle, \hat{\sigma}(x) \geq \frac{\sigma_0}{2}\right\}$. We can show that $E$ is a space of high probability using that \[E^c = \left\{\omega\in\Omega,\exists x\in\cercle, \hat{\sigma}(x) < \frac{\sigma_0}{2} \right\} \subset \left\{\|\hat{\sigma} - \sigma\|_\infty \geq \frac{\sigma_0}{2} \right\},\] and Lemma~\ref{Lemme} (see Section~\ref{Preuve du lemme}) shows that this last set is a set of low probability.\\
We find upper bound of $A_1$ and $A_2$ on the partition $\Omega = E\cup E^c$ and obtain the following result:
\begin{theorem}\label{MISE estimateur seuil}
Suppose $f$ is an element of $\mathbb{L}^2(\cercle)$ and Assumption (A). Then an upper bound of $\hat{f}$ MISE is

\[\Esp(\|\hat{f} - f\|_2^2) \leq \frac{8}{\sigma_0^2}\left(\| \psi - \psi_m\|_2^2 + \frac{\Phi_0^2 D_m \Esp(\Delta_1)}{n}\right)  + \frac{C}{n} ,\]
where $\psi_m$ is the orthogonal projection of $\psi$ on $S_m$,  $C$ is a constant that depends on $\Phi_0, \sigma_0$ and $\|f\|_2$.
\end{theorem}
The proof can be found in Section~\ref{Preuve MISE estimateur seuil}\\
Moreover one can sharpen the result of Theorem~\ref{MISE estimateur seuil}. Here we choose the linear sieves to be the trigonometric spaces (their definitions is reminded in Section~\ref{2.3}). Additionally we suppose that $\psi$ is an element of a Sobolev class $W(\beta,L)$. We recall the definition of the Sobolev class. We define $\alpha_j$ as the following:
\begin{equation}\nonumber
\alpha_j=\left\{\begin{array}{ll}
        j^{\beta}, \hspace{25pt}\text{for even }j,\\
        (j+1)^{\beta}, \hspace{0pt}\text{for odd }j.
    \end{array}
\right.
\end{equation}
We can define the Sobolev class $W(\beta,L)$ for $\beta>0$ and $L>0$ as the following set of functions:
\begin{equation}\label{Sobolev_Def}
W(\beta,L) = \left\{ g \in \mathbb{L}^2(\cercle), \sum_{j=0}^{+\infty} \left(\alpha_j^2 \left|<g,\varphi_j>\right|^2\right) \hspace{3pt}< \frac{L^2}{\pi^{2\beta}} \right\},
\end{equation}
where $\{ \varphi_j\}_{j\in\mathbb{N}}$ is the trigonometric basis of $\mathbb{L}^2(\cercle)$. Finally we recall $\lfloor x \rfloor$ is the floor of $x$.\\
For this class of functions and those linear sieves we can prove the following lemma:
\begin{lemma}\label{Lemme_Sobolev}
Let $\psi$ be an element of $W(\beta_{\psi},L)$ and $S_m$ be the trigonometric space of dimension $D_m$, with $m=m_n$ chosen such that $D_{m_n}=\lfloor n^{\frac{1}{2\beta_{\psi}+1}} \rfloor$. Then we have the following upper bound:
 \[ \|\psi - \psi_m \|_2^2\leq \frac{L^2}{\pi^{2\beta_{\psi}}}n^{\frac{-2\beta_{\psi}}{2\beta_{\psi} +1}}.\]
\end{lemma}
The proof can be found in Section~\ref{Preuve_du_taux_de_hat_psi}.\\
With this lemma, we present a refinement of the previous theorem in this particular setting:
\begin{corollary}\label{Coro estimateur seuil}
Suppose $f$ is an element of $\mathbb{L}^2(\cercle)$, Assumption (A) and that $\psi$ is an element of $W(\beta_{\psi},L)$. Moreover if $S_m$ is a trigonometric space of dimension $D_m$ with $m=m_n$ chosen such that $D_{m_n}=\lfloor n^{\frac{1}{2\beta_{\psi}+1}}\rfloor$ we have the following rate for the MISE of $\hat{f}$:

\[\Esp(\|\hat{f} -f\|_2^2) =  \mathcal{O}\left( n^{\frac{-2\beta_{\psi}}{2\beta_{\psi} +1}} \right).\]
\end{corollary}
The proof can be found in Section~\ref{Proof Coro estimateur seuil}.\\
We observe that the rate of convergence for the MISE of the estimator of $f$ depends only on the regularity of $\psi$. This result means that if $f \in W(\beta_{f},L)$ with $\beta_{\psi} \geq \beta_f$, which happens as soon as $\sigma$ is smooth enough, then the rate of our estimator is the optimal rate of convergence for the estimation of a univariate probability density function belonging to a Sobolev class of regularity $\beta_f$.\\

\subsection{Adaptive estimation procedure}\label{SectionAdaptation}

Since we use a projection estimator for $\psi$, it depends on a parameter $m$. But the $m$ we want to use in Corollary~\ref{Coro estimateur seuil} depends on $\beta_{\psi}$ that is unknown, thus making the right $m$ to choose inaccessible. We then want to implement an automatic data-driven procedure to determine the best parameter $m$ for the estimation. For this we will use a penalization procedure. Meaning we define, for $\mathcal{M}_n$ the set of possible values of $m$ allowed,
\[ \hat{m} = \argmin{m\in \mathcal{M}_n} \left( \gamma_n(\hat{\psi}_m) + \text{pen}(m) \right),\]
where we still have to determine the penalization function $\text{pen}(\cdot)$. For the following results we need another condition on our linear sieves. We suppose our spaces are nested, i.e the mapping $m\mapsto D_m$ is a one-to-one mapping and if $m<m'$ then $S_m \subset S_{m'}$.This is verified for the regular polynomial spaces for dyadic subdivisions, the dyadic wavelet spaces and the trigonometric spaces.\\
The first adaptive procedure we can produce can be found in Proposition~\ref{Penalisation_première} (see Section~\ref{Proof Penalisation_aleatoire}). The problem from this result is that the lower bound we have for $\text{pen}(\cdot)$ depends on an unknown quantity  $\Esp(\Delta) = \Inte \psi(x)\dd x$. Indeed
\[ \Esp(\Delta) = \Esp\left(\mathds{1}_{\{X\in[L,U]\}}\right) = \Inte\Inte\Inte \mathds{1}_{\{x\in [l,u]\}} f(x) \Proba(\dd l,\dd u)\dd x= \Inte \psi(x)\dd x .\]
 There are two ways to solve this. One can either brutally write that $\Esp(\Delta) \leq 1$ since it is the expectation of a random variable that follows a Bernoulli distribution. Or one can estimate this quantity, thus making the penalty term also an estimator, but one needs to verify that the new estimator of $m$ is still adaptive. This is the result of the next theorem:

\begin{theorem}\label{Penalisation_aleatoire}

Suppose the sets $(S_m)_{m\in \mathcal{M}_n}$ are nested, i.e for $m'>m \in \mathcal{M}_n, \\S_m \subset S_{m'}$, $\|\psi\|_\infty < +\infty$ and Assumption (A). For $\hat{m}$ defined as
\[ \hat{m} = \argmin{m\in \mathcal{M}_n} \left( \gamma_n(\hat{\psi}_m) + \widehat{\text{pen}}(m) \right),\]
and
\[ \widehat{\text{pen}}(m) = \kappa \Phi_0^2 \frac{D_m}{n}\left( \frac{1}{n}\sum_{i=1}^n \Delta_i \right),\]
with $\kappa$ universal constant ($\kappa > 8$ works), we obtain the following oracle inequality for the MISE of $\hat{\psi}_{\hat{m}}$:
\[ \Esp\left( \| \hat{\psi}_{\hat{m}} - \psi \|_2^2 \right) \leq  C \inf_{m\in \mathcal{M}_n} \left(\|\psi - \psi_m\|_2^2 + \Phi_0^2 \frac{D_m}{n}\Esp(\Delta)\right) + \frac{\tilde{C}}{n},\]
where $\psi_m$ is the orthogonal projection of $\psi$ on $S_m$, $C$ depends only on $\kappa$ and $\tilde{C}$ is a constant that depends on $\kappa, \Phi_0, \sigma_0$ and $\|\psi\|_\infty$.
\end{theorem}
The proof can be found in Section~\ref{Proof Penalisation_aleatoire}.\\
This means that $\hat{\psi}_{\hat{m}}$ satisfies an oracle inequality, which makes the best  bias-variance tradeoff. If we are in the case of the trigonometric spaces and $\psi$ is a Sobolev function then we obtain the rate of convergence stated in the next result:
\begin{corollary}\label{Vitesse_Adaptative_pour_psi}
Suppose the spaces $(S_m)_{m\in \mathcal{M}_n}$ are the trigonometric spaces, Assumption (A), $\|\psi\|_\infty < +\infty$ and that $\psi$ is an element of $\hspace{2pt} W(\beta_{\psi},L)$. Then we obtain

 \[ \Esp(\|\psi - \hat{\psi}_{\hat{m}} \|_2^2)  = \mathcal{O}\left( n^{\frac{-2\beta_{\psi}}{2\beta_{\psi} +1}} \right).\]
\end{corollary}
This is proved using Theorem~\ref{Penalisation_aleatoire} and Lemma~\ref{Lemme_Sobolev}.\\
This last result tells us that this new estimator is adaptive, meaning it achieves the classical rate of convergence for the estimation of a Sobolev function without the knowledge of its smoothness. 
With this result and using Theorem~\ref{MISE estimateur seuil} we can show that the estimation of $f$ also satisfies an oracle inequality. The next proposition states exactly this:

\begin{proposition}\label{Adaptation seuil}
Suppose the spaces $(S_m)_{m\in \mathcal{M}_n}$ are nested,  $\|f\|_\infty < +\infty$ and Assumption (A).
Moreover we take $\hat{m}$ and $\widehat{\text{pen}}$ defined in Theorem~\ref{Penalisation_aleatoire}.
If $\hat{f}^* = \frac{(\hat{\psi}_{\hat{m}})_+}{\hat{\sigma} \lor n^{-1/2}}$ we obtain the following oracle inequality for the MISE of $\hat{f}^*$:
\[ \Esp\left( \|\hat{f}^* - f \|_2^2 \right) \leq  K\inf_{m\in \mathcal{M}_n} \left(\|\psi - \psi_m\|_2^2 +  \Phi_0^2 \frac{D_m}{n}\Esp(\Delta)\right) + \frac{\tilde{K}}{n},\]
where $\psi_m$ is the orthogonal projection of $\psi$ on $S_m$, $K$ depends only on $\kappa$ and $\tilde{K}$ depends on $\Phi_0, \sigma_0$ and $\|f\|_\infty$.
\end{proposition}
The proof can be found in Section~\ref{Proof Adaptation seuil}.\\
If we are in the case of the trigonometric spaces and $\psi$ is a Sobolev function then we obtain a rate of convergence for $\hat{f}^*$, stated in the next corollary:

\begin{corollary}\label{Coro deuxieme estimateur}
Suppose the spaces $(S_m)_{m\in \mathcal{M}_n}$ are the trigonometric spaces,\\  $\|f\|_\infty < +\infty$, Assumption (A) and $\psi$ is an element of $W(\beta_{\psi},L)$. Then if $\hat{f}^* = \frac{(\hat{\psi}_{\hat{m}})_+}{\hat{\sigma} \lor n^{-1/2}}$ as in Proposition~\ref{Adaptation seuil} we obtain

\[\Esp(\|\hat{f}^* -f\|_2^2) =  \mathcal{O}\left( n^{\frac{-2\beta_{\psi}}{2\beta_{\psi} +1}} \right).\]
\end{corollary}
This is proved using Proposition~\ref{Adaptation seuil} and the same calculations as in Section~\ref{Proof Coro estimateur seuil}.\\
Proposition~\ref{Adaptation seuil} establishes that $\hat{f}^*$ satisfies an oracle inequality, ensuring the best bias-variance tradeoff up to a constant and a remainder term of order $n^{-1}$. Moreover, Corollary~\ref{Coro deuxieme estimateur} states that the data-driven estimator achieves the same convergence rate as the one in Corollary~\ref{Coro estimateur seuil}, without requiring prior knowledge of the smoothness of $\psi$. Finally, if $f \in W(\beta_f, L)$ and $\beta_\psi \geq \beta_f$ (a condition that holds as soon as $\sigma$ is sufficiently smooth) then this convergence rate corresponds to the optimal rate for estimating an univariate density function belonging to a Sobolev class of regularity $\beta_f$.

\section{Simulations}\setcounter{equation}{0}\label{SectionSimulations}
We numerically implement the procedure of Theorem~\ref{Penalisation_aleatoire} to compute $\hat{\psi}_{\hat{m}}$ and then $\hat{f}^*$. We recall the quantities necessary for the estimation. We consider the trigonometric spaces for the linear sieves, thus we know that $D_m = 2m+1$, $\{ \varphi_\lambda\}_{\indices}$ is the trigonometric basis of $S_m$, with $\Lambda_m = \{0,\dots,2m\}$  and $\Phi_0 = \frac{1}{\sqrt{2\pi}}$. Then the set of possible models is represented by $\mathcal{M}_n = \{ 1,\cdots, \lfloor n/2 \rfloor -1 \}$. We then compute for all $m$ in $\mathcal{M}_n$ the penalty term evaluated in $m$ and the contrast evaluated in the estimator $\hat{\psi}_m$.
To compute the penalty terms we use the package CAPUSHE that calibrates the constant $\kappa$ using slope heuristics (see \citelink{Baudry-2012}{Baudry \& al (2012)}). For the contrast term using \eqref{eq:Contraste} and \eqref{eq:Def_hat_psi_m} we can show that 
\[ \gamma_n(\hat{\psi}_m) = -\sum_{\indices} \hat{a}_{\lambda}^2 = - \sum_{\indices} \left(\frac{1}{n}\sum_{i=1}^n  \Delta_i \philambda(X'_i)\right)^2. \]
With that we have
 \[ \hat{m} = \argmin{m\in \mathcal{M}_n} \left(-\sum_{\indices} \hat{a}_{\lambda}^2   + \kappa \frac{2m+1}{2\pi n}\left(\frac{1}{n} \sum_{i=1}^n \Delta_i \right) \right),\]
and it is easily computable. Then we compute $\hat{\psi}_{\hat{m}} = \sum_{\lambda \in \Lambda_{\hat{m}}}\hat{a}_\lambda \philambda$ and finally $\hat{f}^*  = \frac{(\hat{\psi}_{\hat{m}})_+}{\hat{\sigma} \lor n^{-1/2}}$ where $\hat{\sigma}$ is defined in \eqref{eq:def_har_sigma}.\\
To test our estimator we choose to estimate the density probability function of the Von Mises distribution $M(\mu,k)$. We recall the density of this distribution
\[ f(x) = \frac{1}{2\pi I_0(k)}e^{k \cos(x-\mu)},\]
where $\mu$ is the mean direction, $k$ is the concentration parameter and $I_0$ is the modified Bessel function of the first kind of order $0$ such that $f$ is a density on $\cercle$. This distribution is the circular equivalent of the Gaussian distribution on the real line. We consider different types of censorship and use a Monte Carlo method with $N=100$ samples of different sizes to estimate the MISE of our estimator. The models we consider are the following\\
Model $1$: $X\sim M(\pi,1)$, $L  \text{ and } U \text{ are independent}$, $L \sim  M\left(\frac{2\pi}{3},1\right)$ and $U \sim  M\left(\frac{4\pi}{3}, 1\right)$.\\
Model $2$: $X\sim M(\pi,1)$, $L \text{ and } U \text{ are independent}$, $L \sim  M\left(\frac{4\pi}{3},1\right)$ and $U \sim  M\left(\frac{2\pi}{3}, 1\right)$. This is Model $1$ where we exchanged the role of $L$ and $U$ .\\
Model $3$: $X\sim \frac{6}{10}M\left(\frac{\pi}{3},3\right) +  \frac{4}{10}M\left(\frac{15\pi}{9},3\right)$, $L  \text{ and } U \text{ are independent}$, \\$L \sim  M\left(\frac{2\pi}{3},3\right)$ and $U \sim  M\left(\frac{4\pi}{3}, 3\right)$.\\
Model $4$: $X\sim M(\pi,1)$, $L$ and $U$ are deterministic and respectively equal to $\frac{2\pi}{3}$ and $\frac{4\pi}{3}$.\\
 In Table~\ref{table:1} we gather the values of our simulations, which are the estimations of the MISE, the mean percentages of censored values and the mean lengths of the censoring arcs in each simulation. Furthermore, in Figure~\ref{fig:Simu_de_base} we represent in each case a plot of our estimator and the density estimated. \\
\begin{table}[!h]
\begin{center}
\begin{tabular}{|c c c c c|}
\hline\multicolumn{5}{|c|}{MISE estimation}\\
\hline {Sample size} & {$n=50$} & {$n=200$} & {$n=500$} & {$n=1000$} \\
{Model 1} & {$0.056$} & {$0.019$} & {$0.007$} & {$0.004$} \\
{Model 2} & {$0.036$} & {$0.011$} & {$0.005$} & {$0.003$} \\
{Model 3} & {$0.123$} & {$0.072$} & {$0.051$} & {$0.043$} \\
{Model 4} & {$0.183$} & {$0.412$} & {$0.706$} & {$1.211$} \\
\hline \multicolumn{5}{|c|}{$\%$ censored data} \\
\hline {Sample size}  & {$n=50$} & {$n=200$} & {$n=500$} & {$n=1000$} \\
{Model 1} & {$44\%$} & {$44.09\%$} & {$43.787\%$} & {$44.03\%$} \\
{Model 2} & {$55.88\%$} & {$55.95\%$} & {$56.01\%$} & {$56.21\%$} \\
{Model 3} & {$85.32\%$} & {$85.92\%$} & {$86.34\%$} & {$86.03\%$} \\
{Model 4} & {$40.38\%$} & {$38.89\%$} & {$39.37\%$} & {$39.01\%$} \\
\hline\multicolumn{5}{|c|}{Length of censoring arc}\\
\hline {Sample size}  & {$n=50$} & {$n=200$} & {$n=500$} & {$n=1000$} \\
{Model 1} & {$3.49$} & {$3.46$} & {$3.48$} & {$3.48$} \\
{Model 2} & {$2.80$} & {$2.80$} & {$2.82$} & {$2.81$} \\
{Model 3} & {$4.10$} & {$4.09$} & {$4.09$} & {$4.10$} \\
{Model 4} & {$2.09$} & {$2.09$} & {$2.09$} & {$2.09$} \\
\hline
\end{tabular}
\end{center}
\caption{MISE estimation, mean percentage of censored data and mean length of the censoring arc for $N=100$ replications of simulated data of different sample size.}
\label{table:1}
\end{table}
\begin{figure}[!h]
\begin{subfigure}{.45\textwidth}
  \centering
   \includegraphics[width=6.5cm]{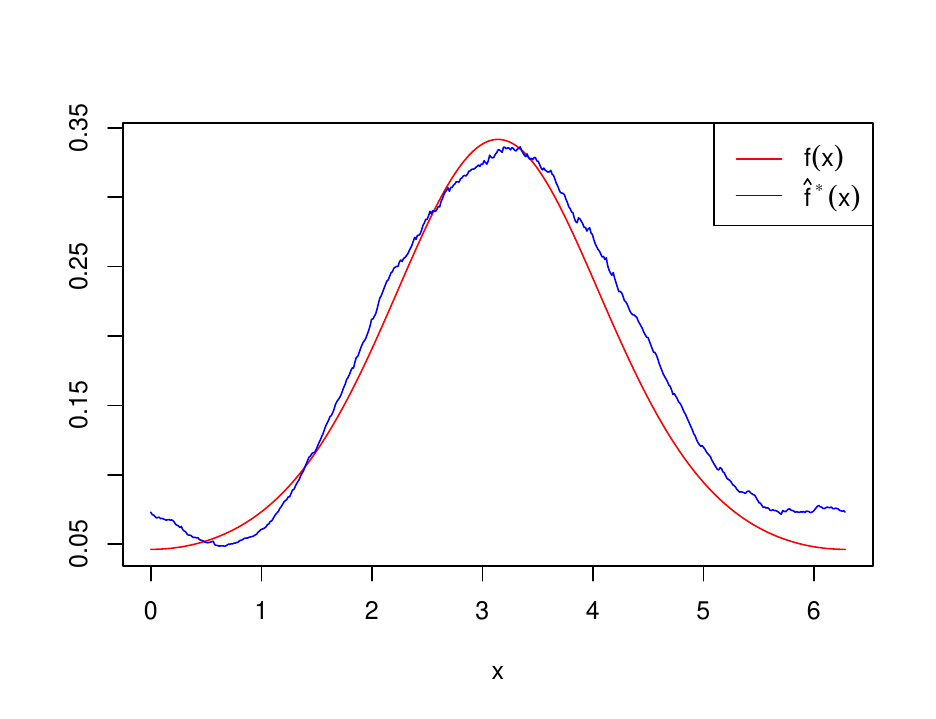}  
  \caption{Model $1$.}
  \label{fig:Sub_1_monitoring}
\end{subfigure}
\begin{subfigure}{.45\textwidth}
  \centering
  \includegraphics[width=6.5cm]{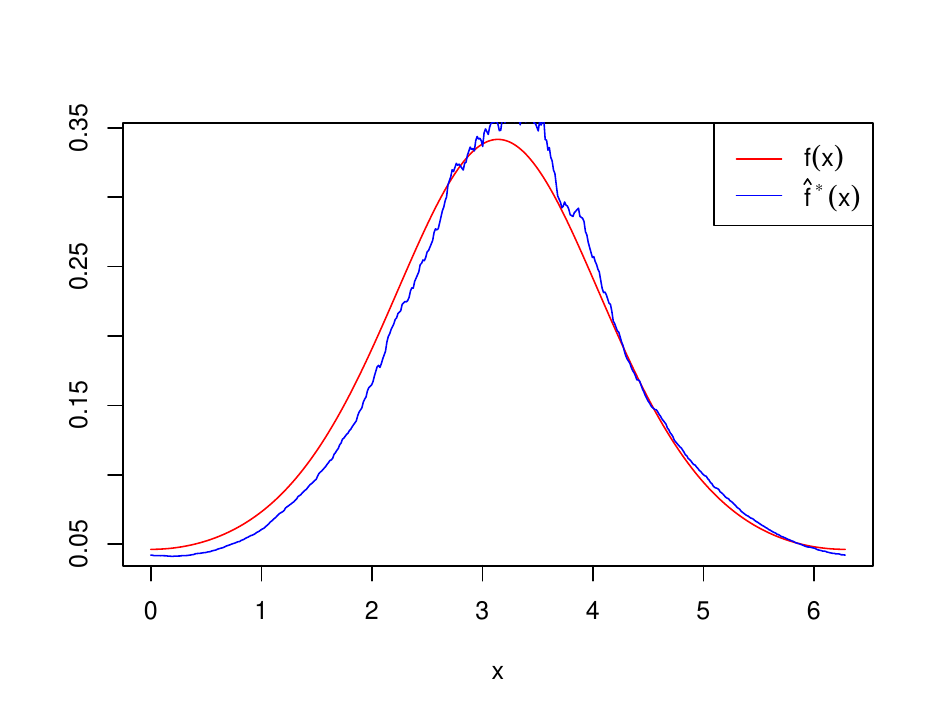}  
  \caption{Model $2$.}
  \label{fig:Sub_1_monitoring}
\end{subfigure}
\begin{subfigure}{.45\textwidth}
  \centering
  \includegraphics[width=6.5cm]{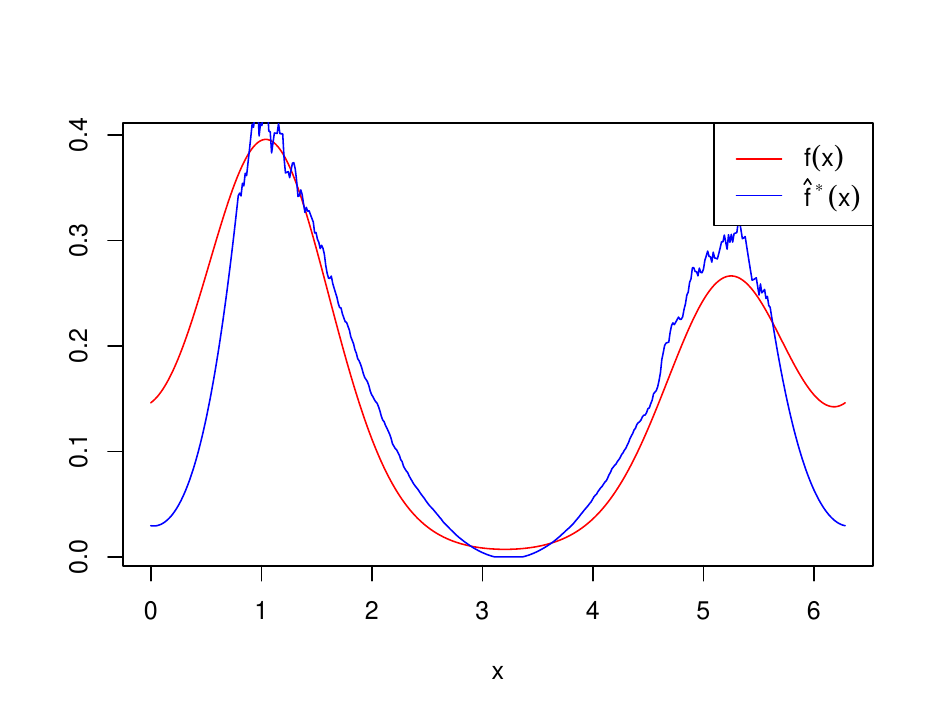}    
  \caption{Model $3$.}
  \label{fig:Sub_1_monitoring}
\end{subfigure}
\begin{subfigure}{.45\textwidth}
  \centering
  \includegraphics[width=6.5cm]{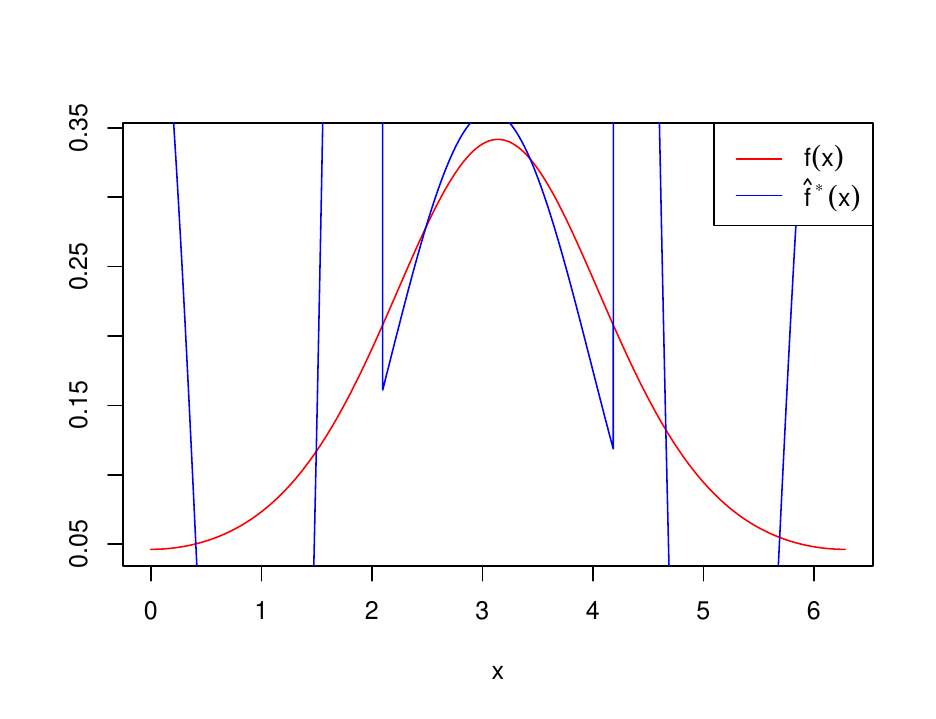} 
  \caption{Model $4$.}
  \label{fig:Sub_1_monitoring}
\end{subfigure}

\caption{Plot of the true density in red and the estimator in blue for a sample of size $n=500$ simulated data. The four plots represent the four different models presented at the beginning of the Section.}
\label{fig:Simu_de_base}
\end{figure}
\begin{figure}[!h]
  \centering
  \includegraphics[width=8cm]{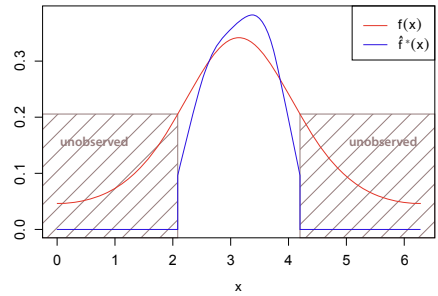} 
\caption{Plot of the true density and $\frac{(\hat{\psi}_{\hat{m}})_+}{\hat{\sigma}}\mathds{1}_{\hat{\sigma}\ne 0}$ for a sample of size $n=500$ simulated data for the Model $4$.}
\label{fig:deterministe_sans_seuil}
\end{figure}
Analyzing Table~\ref{table:1}, we observe that for the first three models, the MISE decreases towards $0$ as $n$ increases, which is consistent with the theoretical results. The third model exhibits a higher MISE due to the strong level of censoring. The last model corresponds to a deterministic case. In that situation, Assumption (A) is not satisfied, since on $[U,L]$, $\sigma$ is equal to $0$. Therefore, the theoretical guarantees no longer apply. In particular, we know that on $[U,L]$, $\hat{\sigma}(x) = 0$, which implies that on this interval, $\hat{f}^* = \sqrt{n} \hat{\psi}_{\hat{m}}$. This explains why the MISE increases with $n$. If it is known from the data that $L$ and $U$ are deterministic, one could instead use the following modified estimator $\frac{\hat{\psi}_{\hat{m}}}{\hat{\sigma}}\mathds{1}_{\{\hat{\sigma}\ne 0\}}$ which avoids boundary effects. Although this modification does not ensure that the MISE converges to $0$, it prevents it from diverging to infinity.\\
The plots in Figure~\ref{fig:Simu_de_base} show that the estimator performs quite satisfactorily for the first three models, even when the level of censoring is high and the underlying distribution is more challenging to estimate, such as in the bimodal case. Moreover, as previously discussed, when the censoring variables are deterministic, the estimation is poor over the censoring arc due to the threshold. Nevertheless, the estimation remains acceptable over the observable part of the circle. The modification we proposed for this specific case is illustrated in Figure~\ref{fig:deterministe_sans_seuil}.\\
Finally, we compare our estimator to the one introduced by  \citelink{Jammalamadaka-2009}{Jammalamadaka \& Mangalam (2009)}. To this end, we replicate their simulation procedure to estimate the parameters of a von Mises distribution. Their approach is parametric: it assumes knowledge of the underlying distributions form and directly estimates its parameters. In contrast, our method is nonparametric: we estimate the density of the observations first, and then use it to infer the parameters of the distribution. We simulate $L$ from a uniform circular distribution, and $U$ is defined as $L - \alpha$, so that the censoring arc always has length $\alpha$. We then generate $200$ samples of size $100$, each drawn from a $M(2,k)$ distribution, where we vary the concentration parameter $k$. The results are presented in Tables~\ref{Tab:Comparaison1} and~\ref{Tab:Comparaison2}.\\
\begin{table}[!h]
\begin{center}
\begin{tabular}{|c c c c c|}
\hline \multirow[t]{2}{*}{ Length of censoring arc $\alpha$ } & \multicolumn{2}{l}{ Estimation of $\mu$} & \multicolumn{2}{l|}{ Estimation of $k$} \\
& $\tilde{\mu}$ & $\hat{\mu}$ & $\tilde{k}$  & $\hat{k}$ \\
\hline 1 & {$1.985$} & {$2.005$} & {$1.047$} & {$1.041$} \\
 3 & {$2.006$}  & {$1.990 $}& {$1.037$} & {$1.021$} \\
\hline
\end{tabular}
\end{center}
\caption{Comparison of the estimation of $(\mu,k)$ when $\mu=2$ and $k=1$, $X \sim M(\mu,k), L$ follows the uniform distribution and $U= L-\alpha$ with $\alpha \in \{1,3\}$.  $(\tilde{\mu},\tilde{k})$ are from Jammalamadaka \& Mangalam and $(\hat{\mu},\hat{k})$ are our estimations.}
\label{Tab:Comparaison1}
\end{table}
\begin{table}[!h]
\begin{center}
\begin{tabular}{|c c c c c|}
\hline \multirow[t]{2}{*}{ Length of censoring arc  $\alpha$ } & \multicolumn{2}{l}{ Estimation of $\mu$} & \multicolumn{2}{l|}{ Estimation of $k$} \\
& $\tilde{\mu}$ & $\hat{\mu}$ & $\tilde{k}$ & $\hat{k}$ \\
\hline 1 & {$1.993$} & {$1.999 $}& {$3.110$} & {$2.936 $}\\
 3 &{$ 2.008$} & {$1.989$}& {$3.083$} & {$3.206$} \\
\hline
\end{tabular}
\end{center}
\caption{Comparison of the estimation of $(\mu,k)$ when $\mu=2$ and $k=3$, $X \sim M(\mu,k), L$ follows the uniform distribution and $U= L-\alpha$ with $\alpha \in \{1,3\}$. $(\tilde{\mu},\tilde{k})$ are from Jammalamadaka \& Mangalam and $(\hat{\mu},\hat{k})$ are our estimations.}
\label{Tab:Comparaison2}
\end{table}
Analyzing Table~\ref{Tab:Comparaison1} and Table~\ref{Tab:Comparaison2}, we notice that our estimations are quite satisfactory even though we do not make the assumption that the distribution is a Von Mises compared to Jammalamadaka and Mangalam.

\section{Proofs}\setcounter{equation}{0}\label{SectionPreuves}

\subsection{A circular interval lemma}\label{Circulaire}

\begin{lemma}\label{EqCirculaire}
If $L \ne U \in \cercle$ and $x\in \cercle$ then we have
\[ \mathds{1}_{\{x\in[L, U]\}} =  \mathds{1}_{\{L\leq x\}} - \mathds{1}_{\{U < x\}} + \mathds{1}_{\{L\geq U\}}.\]
\end{lemma}

\begin{proof}
\begin{align}
\nonumber
\mathds{1}_{\{x\in[L, U]\}} &= \mathds{1}_{\{x\in[L,U]\}}\mathds{1}_{\{L < U\}} + \mathds{1}_{\{x\in[L,U]\}}\mathds{1}_{\{L\geq U\}}\\
\nonumber
&= (\mathds{1}_{\{L\leq x\}} -\mathds{1}_{\{U < x\}})\mathds{1}_{\{L < U\}} + (\mathds{1}_{\{L\leq x\}} +\mathds{1}_{\{x\leq U\}})\mathds{1}_{\{L\geq U\}}\\
\nonumber
&= \mathds{1}_{\{L\leq x\}}(\mathds{1}_{\{L < U\}} + \mathds{1}_{\{L\geq U\}}) - \mathds{1}_{\{U < x\}}(\mathds{1}_{\{L < U\}} + \mathds{1}_{\{L\geq U\}}) + \mathds{1}_{\{L\geq U\}}\\
\nonumber
&= \mathds{1}_{\{L\leq x\}} - \mathds{1}_{\{U < x\}} + \mathds{1}_{\{L\geq U\}}.
\end{align}
\end{proof}

\subsection{Proof of Lemma~\ref{lemme_sur_les_coefs}} \label{Preuve_des_coefs}

We recall the definitions of $\hat{\psi}_m$ and $\psi_m$:
\begin{equation}\nonumber
\left\{\begin{array}{ll}
        \psi_m = \sum_{\indices} a_{\lambda}\philambda \text{, where } a_{\lambda} = <\psi, \philambda> ,\\
        \hat{\psi}_m = \argmin{t\in S_m} \hspace{3pt} \gamma_n(t) = \sum_{\indices} \hat{a}_{\lambda}\philambda \text{,}
    \end{array}
\right.
\end{equation}
where $\gamma_n : t\in \mathbb{L}^2(\cercle) \mapsto \|t\|_2^2 - \frac{2}{n}\sum_{i=1}^n \Delta_i t(X_i)$.\\
Set $\indices$, we prove here that $\hat{a}_\lambda= \tilde{a}_\lambda = \frac{1}{n} \sum_{\indices} \Delta_i \philambda(X'_i)$.\\
For $t$ a function in $S_m$, its representation in the basis $(\philambda)_{\indices}$ is
\begin{equation}\label{eq:t_dans_la_base_de_S_m}
 t = \sum_{\indices} b_{t,\lambda}\philambda \text{, where } b_{t,\lambda} = <t, \philambda>.
\end{equation}
Thus we will consider the restriction of $\gamma_n$ on $S_m$, where we represent every element of $S_m$ with its coordinates in the basis $(\philambda)_{\indices}$. This gives the following definition, for $t \in S_m$ like in \eqref{eq:t_dans_la_base_de_S_m}:
\[ \gamma_n : (b_{t,\lambda})_{\indices} \mapsto \sum_{\indices}\left( b^2_{t,\lambda} - \frac{2 b_{t,\lambda}}{n}\sum_{i=1}^n \Delta_i \philambda(X_i') \right) .\]
Since $\gamma_n$ is a polynomial on its inputs, we compute the partial derivatives. If $b_t=(b_{t,\lambda})_{\indices}$ and $\frac{\partial}{ \partial \lambda}$ is the partial derivative over the $\lambda$-th coordinate, then for $\tau\in\Lambda_m$ we have
\[  \frac{\partial \gamma_n (b_t)}{\partial \tau} = 2b_{t,\tau} - \frac{2}{n}\sum_{i=1}^n \Delta_i \varphi_{\tau}(X_i).\]
Now since $\hat{\psi}_m$ is an extremum of $\gamma_n$, it verifies $\nabla \gamma_n(\hat{\psi}_m) = \nabla \gamma_n (\hat{a}) = 0$, meaning for all $\indices$,
\[ 2\hat{a}_{\lambda} - \frac{2}{n}\sum_{i=1}^n \Delta_i \varphi_{\lambda}(X_i) = 0 \Leftrightarrow \hat{a}_{\lambda} = \frac{1}{n}\sum_{i=1}^n \Delta_i \varphi_{\lambda}(X_i') = \tilde{a}_\lambda. \]
We obtained that $(\tilde{a}_\lambda)_\indices$ are a extremum of $\gamma_n$. We need to prove that it is a minimum of the contrast.
If we look at second order partial derivatives of $\gamma_n$ we have, for $(\tau,\eta)\in\Lambda_m$
\[  \frac{\partial^2 \gamma_n(b_t)}{\partial \tau \partial \eta} = 2 b_{t,\tau} \mathds{1}_{\tau = \eta}. \]
This means that the Hessian matrix of $\gamma_n$ is equal to $2I_{D_m}$ which is a positive semi-definite matrix. It means  that $\gamma_n$ is convex and thus making  $(\tilde{a}_\lambda)_{\indices}$  its unique minimum and we obtain the  announced result.

\subsection{A concentration inequality lemma.}\label{Preuve du lemme}

\begin{lemma}\label{Lemme}
For a sample $(L_i, U_i)_{i\in\{ 1, \dots , n \}}$ i.i.d. with the same law as the random couple $(L, U)$, the function $\sigma : x\in\cercle \mapsto \Proba( x\in [L, U])$ and $\hat{\sigma}$ its empirical estimator. We have for all $y>0$,
\[ \Proba\left(\|\hat{\sigma} - \sigma\|_{\infty} \geq y\right) \leq 6 e^{-2n \frac{y^2}{9}}. \]
\end{lemma}

\begin{proof}

By definition of $\sigma$ and by \eqref{eq:Circulaire} and if we write $V=2\pi - U$ we have, for all $x\in\cercle$,
\begin{align}
\nonumber
 \sigma(x) &= \Proba(x\in [L, U]) \\
\nonumber
&= \Inte\Inte \mathds{1}_{\{x\in [l, u]\}} \Proba_{(L,U)}(\dd l , \dd u) \\
\nonumber
&= \Inte \Inte (\mathds{1}_{\{l\leq x\}} - \mathds{1}_{\{u < x\}} + \mathds{1}_{\{l\geq u\}}) \Proba_{(L, U)}(\dd l , \dd u) \\
\nonumber
&= \Inte \mathds{1}_{\{l\leq x\}} \Proba_L(\dd l) - \Inte \mathds{1}_{\{u < x\}} \Proba_U(\dd u) +\Proba( L \geq U)\\
\nonumber
&=\Inte \mathds{1}_{\{l\leq x\}} \Proba_L(\dd l) - \Inte \mathds{1}_{\{v > 2\pi - x\}} \Proba_V(\dd v) +\Proba( L \geq U)\\
\nonumber
&= \Proba (L \leq x) - (1 - \Proba (V \leq 2\pi - x)) + \Proba( L \geq U)\\
\nonumber
&= \Proba(L \leq x) + \Proba(V \leq 2\pi-x) -1 + (1-\Proba(L < U)) \\
\label{eq:Décompo de phi}
&= F_L (x) + F_V (2\pi-x) - \Proba( L < U),
\end{align}
where $F_L$ and $F_V$ are the distribution functions of $L$ and $V$.\\
For $\hat{\sigma}$ the same calculation shows that
\[ \hat{\sigma}(x) = \hat{F}_L (x) + \hat{F}_V (2\pi-x) - \hat{\Proba} ( L< U),\]
where $\hat{F}_L$ and $\hat{F}_V$ are the empirical distribution function of $L$ and $V$, $\hat{\Proba} ( L< U) = \frac{1}{n}\sum_{i=1}^n \mathds{1}_{\{L_i < U_i\}}$.
Thus we have, for all $x\in\cercle$,
\begin{align*} 
|\sigma(x) - \hat{\sigma}(x)| &\leq |F_L (x) - \hat{F}_L(x) | + |F_V ( 2\pi - x) - \hat{F}_V(2\pi - x) | \\
&+ |\Proba( L < U) -  \hat{\Proba} ( L< U)| \\
&\leq \| F_L - \hat{F}_L \|_\infty + \| F_V - \hat{F}_V \|_\infty + |\Proba( L < U) -  \hat{\Proba} ( L< U)|.
\end{align*}
Thus 
\[ \| \sigma - \hat{\sigma} \|_\infty \leq \| F_L - \hat{F}_L \|_\infty + \| F_V - \hat{F}_V \|_\infty + |\Proba( L < U) -  \hat{\Proba} ( L< U)|.\]
So for $y>0$
\begin{align}
\nonumber
 \Proba\left(\|\hat{\sigma} - \sigma\|_{\infty} \geq y\right) &\leq \Proba\left(\|\hat{F}_L - F_L\|_{\infty} \geq \frac {y}{3}\right) +  \Proba\left(\|\hat{F}_V - F_V\|_{\infty} \geq \frac {y}{3}\right) \\
\nonumber
&+  \Proba\left(|\hat{\Proba} ( L< U) - \Proba( L < U)| \geq \frac {y}{3}\right) \\
& \leq 6 e^{-2n \frac{y^2}{9}},
\end{align}
where the last inequality is obtained by using Dvoretzky–Kiefer–Wolfowitz's inequality on the first two terms and Hoeffding's inequality on the third one.

\end{proof}

\subsection{Proof of Theorem~\ref{MISE estimateur seuil}}\label{Preuve MISE estimateur seuil}

When $\tilde{f}$ is positive we have $\tilde{f} = \hat{f}$. However when $\tilde{f}$ is non positive then $\tilde{f} \leq f$ and $\hat{f}=0$, meaning
\[ |\tilde{f} - f| = f -\tilde{f} \geq |f-0| = |f-\hat{f}|.\]
Put together we have
\[  |\tilde{f} - f| =  |\tilde{f} - f|\mathds{1}_{\tilde{f}>0} +  |\tilde{f} - f|\mathds{1}_{\tilde{f}\leq0} \geq |\hat{f}-f|\mathds{1}_{\tilde{f}>0} +|f-\hat{f}|\mathds{1}_{\tilde{f}<0} = |\hat{f}-f|.\]
We recall $\tilde{f} = \frac{\hat{\psi}_m}{\hat{\sigma} \wedge n^{-1/2}}$  and $f= \frac{\psi}{\sigma}$. Thus we have the following inequality:
\[ \|\tilde{f} - f\|_2^2 \leq 2\underbrace{\left\| \frac{\hat{\psi}_m - \psi}{\hat{\sigma}\lor n^{-1/2}}\right\|_2^2}_{:=A_1} + 2\underbrace{\left\| \frac{\sigma - \hat{\sigma}\lor n^{-1/2}}{\hat{\sigma}\lor n^{-1/2}} f\right\|_2^2}_{:=A_2}. \]
The goal now is to find upper bounds for $A_1$ et $A_2$ using a well found partition of $\Omega$. For this we will use the result of Lemma~\ref{Lemme}.
Define the set $E =  \left\{ \omega\in\Omega,\forall x\in \cercle, \hat{\sigma}(x) \geq \frac{\sigma_0}{2}\right\}$. Notice that \[E^c= \left\{\omega\in\Omega,\exists x\in[0,2\pi[, \hat{\sigma}(x) < \frac{\sigma_0}{2} \right\} \subset \left\{\omega\in\Omega, \|\hat{\sigma} - \sigma\|_{\infty} \geq \frac{\sigma_0}{2} \right\}.\] So thanks to Lemma~\ref{Lemme}, with $y=\frac{\sigma_0}{2}$, we have 
\[ \Proba(E^c) \leq \Proba\left(\left\{\omega\in\Omega, \|\hat{\sigma} - \sigma\|_{\infty} \geq \frac{\sigma_0}{2} \right\} \right) \leq 6e^{-2n \frac{\sigma_0^2}{36}}.\]
We will study how $\Esp(A_1)$ et $\Esp(A_2)$ behave on the partition $\Omega = E\sqcup E^c$. On the set $E^c$ we will use the inequality $\hat{\sigma}\lor n^{-1/2} \geq n^{-1/2}$ and the set $E$ we have the inequality $\hat{\sigma}\lor n^{-1/2} = \hat{\sigma} \geq \frac{\sigma_0}{2}$.\\
Some results are necessary to conclude. We state them in the following propositions.
\begin{proposition}\label{MISE de psi chapeau}
Suppose $\psi\in\mathbb{L}^2(\cercle).$ An upper bound for the MISE of the estimator $\hat{\psi}_m$ is the following:
\[ \Esp\left(\|\psi - \hat{\psi}_m \|_2^2\right)\leq \| \psi - \psi_m\|_2^2 + \frac{\Phi_0^2 D_m \Esp(\Delta)}{n}.\]

\end{proposition}
The proof can be found in Section~\ref{Proof of première prop}.
\begin{proposition}\label{MISE de phi chapeau}
Suppose $f\in\mathbb{L}^2(\cercle)$. An upper bound for the weighted MISE of the estimator $\hat{\sigma}$ is the following:
\[ \Esp\left(\|(\sigma - \hat{\sigma})f\|_2^2\right)\leq \frac{\|f\|_2^2}{4n}.\]
\end{proposition}
The proof can be found in Section~\ref{Preuve_du_taux_de_hat_phi}.
\begin{proposition}\label{Troisième majoration}
Recall $E^c= \left\{\omega\in\Omega,\exists x\in[0,2\pi[, \hat{\sigma}(x) < \frac{\sigma_0}{2} \right\}$. Suppose $\psi\in\mathbb{L}^2(\cercle)$ and Assumption (A) . An upper bound for the MISE of the estimator $\hat{\psi}_m$ on the set $E^c$ is the following:
\[\Esp\left(\| \hat{\psi}_m - \psi\|_2^2 \mathds{1}_{E^c}\right)\leq (n \Phi_0^2 + \|\psi\|_2^2)12 e^{-2n \frac{\sigma_0^2}{36}}.\]
\end{proposition}
The proof can be found in Section~\ref{Preuve de la troisième majoration}.\\

For $A_1$ we have
\begin{align*}
\Esp(A_1\mathds{1}_{E})=\Esp\left(\left\| \frac{\hat{\psi}_m - \psi}{\hat{\sigma}\lor n^{-1/2}}\right\|_2^2\mathds{1}_{E} \right) &\leq \frac{4}{\sigma_0^2}
\Esp\left(\|\hat{\psi}_m - \psi\|_2^2\right) \\
&\leq \frac{4}{\sigma_0}\left(\| \psi - \psi_m\|_2^2 + \frac{\Phi_0^2 D_m \Esp(\Delta)}{n} \right),
\end{align*}
using Proposition~\ref{MISE de psi chapeau}, and
\begin{align*} 
\Esp(A_1\mathds{1}_{E^c})=\Esp\left(\left\| \frac{\hat{\psi}_m - \psi}{\hat{\sigma}\lor n^{-1/2}}\right\|_2^2\mathds{1}_{E^c} \right) &\leq n \Esp\left(\| \hat{\psi}_m - \psi\|_2^2 \mathds{1}_{E^c}\right) \\
&\leq 2n\left(n\Phi_0^2 + \|\psi\|_2^2\right)\Proba(E^c) \\
&\leq \left(n\Phi_0^2 + \|\psi\|_2^2\right)12ne^{-2n \frac{\sigma_0^2}{36}},
\end{align*}
using Proposition~\ref{Troisième majoration} and Lemma~\ref{Lemme}.\\


For $A_2$ we have
\[\Esp(A_2\mathds{1}_{E})=\Esp\left(\left\| \frac{(\sigma - \hat{\sigma}\lor n^{-1/2})f}{\hat{\sigma}\lor n^{-1/2}}\right\|_2^2 \mathds{1}_{E} \right) \leq \frac{4}{\sigma_0^2}\Esp\left(\|(\sigma - \hat{\sigma}) f\|_2^2\right) \leq \frac{\|f\|_2^2}{\sigma_0^2 n}, \]
using Proposition~\ref{MISE de phi chapeau}, and
\begin{align*} 
\Esp(A_2\mathds{1}_{E^c})=\Esp\left(\left\| \frac{(\sigma - \hat{\sigma}\lor n^{-1/2})f}{\hat{\sigma}\lor n^{-1/2}}\right\|_2^2\mathds{1}_{E^c} \right) &\leq n \Esp\left(\|\underbrace{(\sigma - \hat{\sigma}\lor n^{-1/2})}_{\leq 2}f\|_2^2 \mathds{1}_{E^c}\right)\\
&\leq 4n\|f\|_2^2\Proba(E^c) \\
&\leq \left(2\|f\|_2^2\right) 12ne^{-2n \frac{\sigma_0^2}{36}},
\end{align*}
using Lemma~\ref{Lemme}.\\
With the different inequalities obtained we have
\begin{align*}
\Esp\left(\|\hat{f} - f\|_2^2\right) &\leq \frac{1}{\sigma_0^2}\left(8\| \psi - \psi_m\|_2^2 + \frac{8\Phi_0^2 D_m \Esp(\Delta)}{n} + \frac{2\|f\|_2^2}{ n}\right)  \\
&+ \left(n\Phi_0^2 + \|\psi\|_2^2 + 2\|f\|_2^2\right)24ne^{-2n \frac{\sigma_0^2}{36}}\\
&\leq  \frac{8}{\sigma_0^2}\left(\| \psi - \psi_m\|_2^2 + \frac{\Phi_0^2 D_m \Esp(\Delta)}{n}\right) + \frac{C}{n},
\end{align*}
where $C$ is a constant that depends on $\Phi_0, \sigma_0$ and $\|f\|_2$. We obtained the inequality announced in the theorem.

\subsection{Proof of Proposition~\ref{MISE de psi chapeau}}\label{Proof of première prop}

Using Pythagorean theorem, \eqref{eq:Def_hat_psi_m} and \eqref{eq:Def_psi_m}, we have
\begin{align*}
    \|\psi - \hat{\psi}_m \|_2^2 &= \| \psi - \psi_m\|_2^2 + \|\psi_m - \hat{\psi}_m\|_2^2 \\
    &= \|\psi - \psi_m\|_2^2 + \sum_{\indices} (a_{\lambda} - \hat{a}_{\lambda})^2\\
    &= \| \psi - \psi_m\|_2^2 + \sum_{\indices} \left(\frac{1}{n}\sum_{i=1}^n \Delta_i\philambda(X_i) - \underbrace{\Inte\philambda(x)\psi(x) \dd x}_{=\Esp\left(\frac{1}{n}\sum_{i=1}^n \Delta_i\philambda(X_i')\right)} \right)^2.
\end{align*}
Using \eqref{eq:Linear_Sieves2}, an upper bound of $\hat{\psi}_m$ MISE is 
\begin{align*}
    \Esp\left(\|\psi - \hat{\psi}_m \|_2^2\right) &=\| \psi - \psi_m\|_2^2 + \sum_{\indices} \Var\left(\frac{1}{n}\sum_{i=1}^n \Delta_i\philambda(X_i')\right)\\
    &=\| \psi - \psi_m\|_2^2 + \frac{1}{n}\sum_{\indices}\Var(\Delta \philambda(X'))\\
	&\leq \| \psi - \psi_m\|_2^2 + \frac{1}{n}\sum_{\indices} \Esp\left(\Delta^2 \philambda^2(X')\right)\\
	&\leq \| \psi - \psi_m\|_2^2 + \frac{1}{n} \Esp\left(\|\sum_{\indices}\philambda^2 \|_{\infty}\Delta\right)\\
	&\leq \| \psi - \psi_m\|_2^2 + \frac{\Phi_0^2 D_m \Esp(\Delta)}{n}.
\end{align*}
We obtain the announced result.

\subsection{Proof of Proposition~\ref{MISE de phi chapeau}.}\label{Preuve_du_taux_de_hat_phi}

We recall from \eqref{eq:def_har_sigma} that \\$\hat{\sigma}(x) = \frac{1}{n}\sum_{i=1}^n \mathds{1}_{\{x\in[L_i, U_i]\}}$ and that it is the empirical estimator of $\sigma$, such that $\Esp(\hat{\sigma}) = \sigma$. Then we have
\begin{align*}
\Esp\left(\|(\hat{\sigma} - \sigma)f\|_2^2\right) &= \Esp\left(\Inte (\hat{\sigma}(x) -\sigma(x))^2 f^2(x)\dd x\right) \\
&= \Inte \Var(\hat{\sigma}(x))f^2(x) \dd x = \Inte \frac{1}{n} \Var\left(\mathds{1}_{\{x\in[L, U]\}}\right)f^2(x)\dd x \\
&= \Inte \frac{1}{n} \left( \Esp\left(\mathds{1}_{\{x\in[L, U]\}}^2\right) - \Esp\left(\mathds{1}_{\{x\in[L, U]\}}\right)^2 \right)  f^2(x)\dd x \\
&= \Inte \frac{1}{n} \left(\sigma(x) - \sigma^2(x)\right)f^2(x) \dd x  \leq \Inte \frac{f^2(x)}{4n} \dd x = \frac{\|f\|_2^2}{4n}.
\end{align*}
We obtain the announced inequality.

\subsection{Proof of Proposition~\ref{Troisième majoration}}\label{Preuve de la troisième majoration}

First we have the following decomposition:
\begin{align}
\nonumber\|\hat{\psi}_m - \psi \|_2^2 &\leq 2 (\|\hat{\psi}_m\|_2^2 + \|\psi\|_2^2).
\end{align}
To find an upper bound of the $\mathbb{L}^2$ norm of $\hat{\psi}_m$ we use the property of the sieves stated in \eqref{eq:Linear_Sieves2}.
\begin{align} \|\hat{\psi}_m\|_2^2 &\nonumber=\sum_{\indices} \hat{a}_{\lambda}^2 \leq  \sum_{\indices} \left( \frac{1}{n}\sum_{i=1}^n \Delta_i \philambda(X_i')\right)^2 \leq \sum_{\indices} \frac{n}{n^2}\sum_{i=1}^n\philambda^2(X_i)\\
& \label{eq:Ineg9.9} \leq \frac{1}{n}\sum_{i=1}^n\sum_{\indices}\philambda^2(X_i)  \leq  \frac{1}{n}\sum_{i=1}^n \|\sum_{\indices}\philambda^2\|_{\infty} \leq \|\sum_{\indices}\philambda^2\|_{\infty} \leq \Phi_0^2D_m \leq n \Phi_0^2 .
\end{align}
So we have
\[\Esp\left(\| \hat{\psi}_m - \psi\|_2^2 \mathds{1}_{E^c}\right)\leq 2\Esp\left( (n \Phi_0^2 + \|\psi\|_2^2)\mathds{1}_{E^c}\right) \leq 2\left(n \Phi_0^2 + \|\psi\|_2^2\right)\Proba(E^c).\]
What remains is to find an upper bound of $\Proba(E^c)$.\\
Using its definition, we have $\Proba(E^c) \leq \Proba\left(\|\hat{\sigma} - \sigma\|_{\infty} \geq \frac{\sigma_0}{2}\right)$. Using Lemma~\ref{Lemme}, with $y=\frac{\sigma_0}{2}$, we obtain the following upper bound:
\[\Esp\left(\|\hat{\psi}_m - \psi\|_2^2\mathds{1}_{E^c}\right)\leq \left(n \Phi_0^2 + \|\psi\|_2^2\right)12 e^{-2n \frac{\sigma_0^2}{36}}.\]
We find the inequality of the proposition.

\subsection{Proof of Lemma~\ref{Lemme_Sobolev}}\label{Preuve_du_taux_de_hat_psi}

For $m \in \mathcal{M}_n$, the trigonometric space $S_m$ is generated by $\{\varphi_0 := \frac{1}{\sqrt{2\pi}}, \varphi_j := \frac{1}{\sqrt{\pi}}\cos(j\cdot),\varphi_{j+1}:=\frac{1}{\sqrt{\pi}}\sin(j\cdot) | \text{ for } j \in \{ 1, \dots , m \}\}$. Its dimension is $D_m = 2m+1$ and $\Phi_0 = \frac{1}{\sqrt{2\pi}}$. For all $j\in \mathbb{N}$ we recall $a_j= \Inte \psi(x)\varphi_j(x) \dd x$. With this orthonormal basis of $\mathbb{L}^2(\cercle)$ and $S_m$ we can write those equalities in the $\mathbb{L}^2$ sense,
\[\psi = \sum_{j=0}^{+\infty} a_j \varphi_j, \hspace{40pt} \psi_m = \sum_{j=0}^{D_m}  a_j \varphi_j .\]
Using those definitions we have
\[\| \psi - \psi_m\|_2^2\leq \sum_{j= D_m + 1}^{+\infty}a_j^2 . \]
Now we use the fact that $\psi$ is an element of the Sobolev class $W(\beta_{\psi},L)$. This means that
\[ \sum_{j=0}^{+\infty} \left(\alpha_j^2 \left|<\psi,\varphi_j>\right|^2\right) = \sum_{j=0}^{+\infty} \alpha_j^2 a_j^2  \hspace{3pt}< \frac{L^2}{\pi^{2\beta}}, \]
where the $\alpha_j$ are defined as the following:
\begin{equation}\nonumber
\alpha_j=\left\{\begin{array}{ll}
        j^{\beta_{\psi}}, \hspace{25pt}\text{for even }j,\\
        (j-1)^{\beta_{\psi}}, \hspace{0pt}\text{for odd }j.
    \end{array}
\right.
\end{equation}
Because the $\alpha_j$ are non-decreasing we also have
\[ \sum_{j= D_m + 1}^{+\infty}a_j^2 \leq \frac{1}{\alpha_{D_m +1}^2}\sum_{j= 0}^{+\infty}\alpha_j^2 a_j^2 \leq \frac{L^2}{\pi^{2\beta_{\psi}}\alpha_{D_m +1}^2}  = \frac{L^2}{\pi^{2 \beta_{\psi}}} \left(D_m + 1\right)^{-2\beta_{\psi}} \leq \frac{L^2}{\pi^{2\beta_{\psi}}}n^{\frac{-2\beta_{\psi}}{2\beta_{\psi} +1}}.\]
which is exactly the inequality we wanted to prove.

\subsection{Proof of Corollary~\ref{Coro estimateur seuil}}\label{Proof Coro estimateur seuil}
We remind that Theorem~\ref{MISE estimateur seuil} gives us
\begin{equation}
\nonumber
\Esp\left(\|\hat{f} - f\|_2^2\right) \leq \frac{8}{\sigma_0^2}\left(\| \psi - \psi_m\|_2^2 + \frac{\Phi_0^2 D_m \Esp(\Delta)}{n}\right)  + \frac{C}{n}.
\end{equation}
We know from Lemma~\ref{Lemme_Sobolev} that
\[ \|\psi - \psi_m\|_2^2 \leq \frac{L^2}{\pi^{2\beta_\psi}}n^{\frac{-2\beta_\psi}{2\beta_\psi + 1}}.\]
Since $D_m = \lfloor n^{\frac{1}{2\beta_{\psi}+1}}\rfloor$ and $\Phi_0=\frac{1}{\sqrt{2\pi}}$ in the case of trigonometric spaces we have
\[ \frac{\Phi_0^2 D_m \Esp(\Delta)}{n} \leq \frac{1}{2\pi}n^{\frac{1}{2\beta_{\psi}+1} - 1} \leq \frac{1}{2\pi}n^{\frac{-2\beta_\psi}{2\beta_\psi + 1}}.\]
Combining those two results we obtain
\[ \Esp\left(\|\hat{f} - f\|_2^2\right) = \mathcal{O}\left(n^{\frac{-2\beta_{\psi}}{2\beta_{\psi} +1}}\right). \]
We obtain the rate of convergence mentioned in the corollary.

\subsection{A useful Talagrand's inequality}

\begin{lemma}\label{Talagrand}
Set $Z_1,\dots, Z_n$ independent random variables and \\$\nu_n(l)= \frac{1}{n} \sum_{i=1}^n \left(l(Z_i) - \Esp(l(Z_i)) \right)$ for $l$ an element of a countable class $\mathcal{L}$ of uniformly bounded measurable functions. Then
\[ \Proba\left(\sup_{l\in\mathcal{L}} \left|\nu_n(l)\right| \geq H + \xi\right) \leq \exp\left(-\frac{n\xi^2}{2(v+4HM_1) + 6M_1\xi} \right),\]
for \hspace{0.2cm}$\sup_{l\in\mathcal{L}} \|l\|_{\infty} \leq M_1 , \hspace{0.2cm}\Esp\left(\sup_{l\in\mathcal{L}} \left|\nu_n(l)\right|\right) \leq H , \hspace{0.2cm}\sup_{l\in\mathcal{L}}\frac{1}{n} \sum_{i=1}^n\Var(l(Z_i)) \leq v$.\\
Moreover, for $\epsilon >0$ we have
\[ \Esp\left( \sup_{l\in\mathcal{L}} \left|\nu_n(l)\right|^2 - 2(1+2\epsilon)H^2\right)_+ \leq C\left(\frac{v}{n}e^{-K_1 \epsilon \frac{nH^2}{v}} + \frac{M_1^2}{n^2 C(\epsilon)^2}e^{-K_2C(\epsilon)\sqrt{\epsilon}\frac{nH}{M_1}} \right),\]
where $C(\epsilon) = \left(\sqrt{1+\epsilon}-1\right)\wedge 1$ (writing $a \wedge b := \min(a,b)$), $K_1= 1/6$ and $K_2 = 1/(21\sqrt{2})$.
\end{lemma}
The first inequality is a result from \citelink{Talagrand}{Talagrand ('96)} and \citelink{Klein-Rio}{Klein, Rio (2005)}.\\
We will show how we obtain the second inequality.
For this we use an idea from \citelink{Birge-Massart-1998}{Birgé, Massart ('98)}. Setting $\xi = \lambda + \eta H$ and $\frac{a+b}{c+d+e} \geq \frac{1}{3}\left(\frac{a}{c} \wedge \frac{b}{d} \wedge \frac{a}{e}\right)$ for $a,b,c,d,e \in \mathbb{R}_+$ we have
\begin{align*}
\frac{\xi^2}{2v + 8HM_1+ 6M_1 \xi} &\geq \frac{ \lambda^2 + 2\eta\lambda H}{2v + 8HM_1 + 6M_1\lambda + 6M_1 \eta H}\\
&\geq \frac{1}{3}\left(\frac{\lambda^2}{2v}\wedge \frac{\lambda\eta}{M_1\left(4 + 3\eta\right)} \wedge \frac{\lambda^2}{6M_1\lambda}\right)\\
&\geq \frac{1}{3}\left(\frac{\lambda^2}{2v}\wedge \frac{(\eta\wedge 1)\lambda}{7M_1}\right),
\end{align*}
using that $\frac{\eta}{4+3\eta} \geq \frac{\eta \wedge 1}{7}$.
With this inequality and Talagrand's inequality we obtain
\begin{equation}\label{eq:Talgrand_utile}
\Proba\left(\sup_{l\in\mathcal{L}} \left|\nu_n(l)\right| \geq (1+\eta)H + \lambda\right)\leq \exp\left( -\frac{n}{3}\left(\frac{\lambda^2}{2v}\wedge \frac{(\eta\wedge 1)\lambda}{7M_1}\right)\right).
\end{equation}
For $\epsilon>0$ fixed, we set $\eta = \sqrt{1+\epsilon} -1 = C(\epsilon)$, such that $\sqrt{1+\epsilon}=1+\eta$ and $\lambda= \sqrt{\epsilon H^2 +\frac{t}{2}}$ for a given $t$ and we use the fact that $\Esp(X)_+ = \int_0^{\infty} \Proba(X\geq t) \dd t$ with $X= \sup_{l\in\mathcal{L}} \left|\nu_n(l)\right|^2 - 2(1+\epsilon)H^2$. Using \eqref{eq:Talgrand_utile} we obtain

\begin{align*}
&\Esp\left( \sup_{l\in\mathcal{L}} \left|\nu_n(l)\right|^2 -   2(1+2\epsilon)H^2\right)_+ \\
&= \int_0^{\infty} \Proba\left(\sup_{l\in\mathcal{L}} \left|\nu_n(l)\right|^2 - 2(1+2\epsilon)H^2 \geq t \right) \dd t \\
&= \int_0^{\infty} \Proba\left(\sup_{l\in\mathcal{L}} \left|\nu_n(l)\right|^2 \geq 2(1+2\epsilon)H^2 + t\right) \dd t \\
&= \int_0^{\infty} \Proba\left(\sup_{l\in\mathcal{L}} \left|\nu_n(l)\right| \geq \sqrt{2(1+2\epsilon)H^2 + t} \right) \dd t \\
&= \int_0^{\infty} \Proba\left(\sup_{l\in\mathcal{L}} \left|\nu_n(l)\right| \geq \sqrt{2(1+\epsilon)H^2 + t + 2\epsilon H^2} \right) \dd t \\
&\leq \int_0^{\infty} \Proba\left(\sup_{l\in\mathcal{L}} \left|\nu_n(l)\right| \geq \sqrt{(1+\epsilon)}H + \sqrt{\frac{t}{2} + \epsilon H^2} \right) \dd t \\
&\leq \int_{0}^{\infty} \exp\left( -\frac{n}{3}\left( \frac{\epsilon H^2 +\frac{t}{2}}{2v} \wedge \frac{C(\epsilon)\sqrt{\epsilon H^2 +\frac{t}{2}}}{7M_1}\right) \right) \dd t \\
&\leq e^{-\frac{n\epsilon H^2}{6v}}\int_{0}^{\infty} e^{-\frac{nt}{12v}} \dd t + e^{-\frac{nC(\epsilon)\sqrt{\epsilon}H}{21\sqrt{2}M_1}}\int_0^{\infty}e^{-\frac{nC(\epsilon)\sqrt{t}}{42M_1}}\dd t \\
&\leq  e^{-K_1\frac{n\epsilon H^2}{v}} \cdot 12\frac{v}{n} + e^{-K_2\frac{nC(\epsilon)\sqrt{\epsilon}H}{M_1}} \cdot 2\left(\frac{42M_1}{nC(\epsilon)}\right)^2\\
&\leq C\left(\frac{v}{n}e^{-K_1\frac{n\epsilon H^2}{v}} + \frac{M_1^2}{n^2 C(\epsilon)^2}e^{-K_2\frac{nC(\epsilon)\sqrt{\epsilon}H}{M_1}} \right),
\end{align*}
with $C = 2*42^2 = 3528$ we obtain the announced result.

\subsection{First adaptive property}\label{Proof Penalisation_première}
We can prove the following adaptive property for our estimator:
\begin{proposition}\label{Penalisation_première}
Suppose the spaces $(S_m)_{m\in \mathcal{M}_n}$ are nested, i.e for $m'>m \in \mathcal{M}_n, S_m \subset S_{m'}$ and $\|\psi\|_\infty < +\infty$. Then for $\hat{m}$ defined as
\[ \hat{m} = \argmin{m\in \mathcal{M}_n} \left( \gamma_n(\hat{\psi}_m) + \text{pen}(m) \right),\]
and 
\[ \text{pen}(m) \geq \kappa \Phi_0^2 \frac{D_m}{n} \Esp(\Delta),\]
with $\kappa$ universal constant ($\kappa > 4$ works), we obtain the following inequality for the MISE of $\hat{\psi}_{\hat{m}}$:
\[ \Esp\left( \| \hat{\psi}_{\hat{m}} - \psi \|_2^2 \right) \leq  \breve{C}\inf_{m\in \mathcal{M}_n} \left(\|\psi - \psi_m\|_2^2 + \text{pen}(m) \right) + \frac{\mathring{C}}{n},\]
where $\breve{C}$ depends only on $\kappa$ and $\mathring{C}$ is a constant that depend on $\kappa$, $\Phi_0$ and $\|\psi\|_\infty$.
\end{proposition}
First we need the following lemma:
\begin{lemma}\label{Lemme_de_Talagrand}
Let the set $B_{m,m'}(0,1) := \{ t\in S_m + S_{m'}, \|t\|_2=1 \}$. Writing $l_t(z)=l_t(x,\delta) := \delta t(x)$ with $t$ in $B_{m,m'}(0,1)$ and $Z_i = (X_i,\Delta_i)$, we define the following empirical process,
\[ \nu_n(l_t) = \frac{1}{n} \sum_{i=1}^n \left( \Delta_i t(X_i) - \Esp(\Delta_i t(X_i)) \right).\]
We have the following inequality for every $\epsilon > 0$:
\[ \Esp \left( \sup_{t\in B_{m,m'}(0,1)} \nu_n^2(l_t) - \text{p}(m,m') \right)_+ \leq \frac{\kappa_1}{n} \left( e^{-\kappa_2 \epsilon (D_m + D_{m'})} + \frac{e^{-\kappa_3 C(\epsilon)\sqrt{\epsilon}\sqrt{n}}}{C(\epsilon)^2} \right),\]
with $\text{p}(m,m') = 2(1+2\epsilon)\frac{\Phi_0^2(D_m + D_{m'})}{n}\Esp(\Delta)$ and $C(\epsilon) = (\sqrt{1+\epsilon} - 1) \wedge 1$. The constants $\kappa_1,\kappa_2,\kappa_3$ depend on $\Phi_0$ and $\psi$. 
\end{lemma}
We shall prove this lemma.
\begin{proof}
This lemma is an application of Talagrand's inequality reminded in \\Lemma~\ref{Talagrand}.
To use it we need to find three upper bounds:
\[ \sup_{l\in\mathcal{L}} \| l\|_{\infty} \leq M_1, \hspace{40pt} \sup_{l\in \mathcal{L}} \Var(l(Z_1)) \leq v, \hspace{40pt}  \Esp\left( \sup_{l\in\mathcal{L}} \left|\nu_n(l)\right| \right) \leq H, \]
where, thanks to usual density arguments, instead of a countable class, we are allowed to take a unit ball in a finite functional space. So we set $\mathcal{L} = \{l_t, t \in B_{m,m'}(0,1)\}$ and its elements are $ l_t(x,\delta) := \delta t(x)$ and the random variables are $Z_i = (X_i, \Delta_i)$ for $i\in \{ 1, \dots , n \}$. Since our spaces are nested we know that $S_m + S_{m'} = S_{\max(m,m')}$ and thus if we write $D(m,m')$ the dimension of $S_m + S_{m'}$ we have $D(m,m') = D_{m \vee m'} \leq D_m + D_{m'}$ where $a\vee b = \max(a,b)$.
First we have
\begin{align*}
\sup_{l\in\mathcal{L}} \| l\|_{\infty} &= \sup_{t\in B_{m,m'}(0,1)} \| l_t\|_{\infty}\leq \sup_{t\in B_{m,m'}(0,1)} \| t \|_{\infty}\\
& \leq\sup_{t\in B_{m,m'}(0,1)} \Phi_0\sqrt{D(m,m')} \|t\|_2 \leq \Phi_0\sqrt{D_m + D_{m'}} =: M_1 ,
\end{align*}
using \eqref{eq:Linear_Sieves}.
Then we have
\begin{align*}
\sup_{l\in \mathcal{L}} \Var(l(Z_1)) &= \sup_{t\in B_{m,m'}(0,1)} \Var(l_t(X_1,\Delta_1))\leq \sup_{t\in B_{m,m'}(0,1)} \Var(\Delta_1t(X_1))\\
& \leq \sup_{t\in B_{m,m'}(0,1)} \Esp\left(\Delta_1 t^2(X_1)\right)\\
&\leq \sup_{t\in B_{m,m'}(0,1)} \Inte\Inte\Inte \mathds{1}_{\{x\in [l,u]\}} t^2(x)f(x) \Proba_{(L,U)}(\dd l, \dd u) \dd x \\
&\leq \sup_{t\in B_{m,m'}(0,1)} \Inte t^2(x)f(x)\sigma(x) \dd x \leq \sup_{t\in B_{m,m'}(0,1)} \Inte t^2(x) \psi(x)\dd x \\
&\leq \|\psi\|_{\infty} =: v .
\end{align*}
We notice $\nu_n(l_t) = \frac{1}{n}\sum_{i=1}^n \left(\Delta_i t(X_i) - \Esp(\Delta_i t(X_i)\right) = <t , \hat{\psi}_{m \vee m'} - \psi_{m \vee m'}>$. Moreover,  if we write $\Lambda_{m,m'}$ the indexes of a basis of $S_m+S_{m'} = S_{m\vee m'}$, we know that 
\begin{align*}
 \sum_{\lambda \in \Lambda_{m,m'}} \nu_n^2(l_{\philambda}) =\|\hat{\psi}_{m \vee m'} - \psi_{m \vee m'}\|_2^2  &= \sup_{t\in B_{m,m'}(0,1)} \left|<t,\hat{\psi}_{m \vee m'} - \psi_{m \vee m'}>\right|^2 \\
 &=  \sup_{t\in B_{m,m'}(0,1)}  \left|\nu_n(l_t)\right|^2,
\end{align*}
 and 
 \[\Esp\left(\sum_{\lambda \in \Lambda_{m,m'}} \nu_n^2(l_{\philambda})\right) = \sum_{\lambda \in \Lambda_{m,m'}} \Var\left(\frac{1}{n}\sum_{i=1}^n\Delta_i \philambda(X_i)\right).\] Using those two equalities and \eqref{eq:Linear_Sieves2} we have
\begin{align*}
\Esp\left( \sup_{l\in\mathcal{L}} \left|\nu_n(l)\right|^2 \right) &= \Esp\left( \sup_{t\in B_{m,m'}(0,1)} \left|\nu_n(l_t)\right|^2 \right) = \frac{1}{n} \sum_{\lambda \in \Lambda_{m,m'}} \Var(\Delta_1\philambda(X_1))\\ 
& \leq \frac{1}{n} \sum_{\lambda \in \Lambda_{m,m'}} \Esp\left( \Delta_1\philambda^2(X_1)\right)\leq \frac{1}{n} \sum_{\lambda \in \Lambda_{m,m'}} \Inte \philambda^2(x) \psi(x) \dd x \\
&\leq \frac{1}{n} \Inte \| \sum_{\lambda \in \Lambda_{m,m'}} \philambda^2\|_{\infty} \psi(x) \dd x \leq \frac{1}{n} \Inte \Phi_0^2 D(m,m') \psi(x) \dd x \\
& \leq \frac{\Phi_0^2 (D_m + D_{m'})}{n}\Esp(\Delta) =: H^2.
\end{align*}
If we set $\text{p}(m,m') = 2(1+2\epsilon)\Phi_0^2\frac{(D_m + D_{m'})}{n}\Esp(\Delta)$ then by Talagrand's inequality recalled in Lemma~\ref{Talagrand} we have the desired inequality.
\end{proof}

Let us now prove the Proposition~\ref{Penalisation_première} itself.\\
Set $m\in \mathcal{M}_n$.
By definition of $\hat{m}$ and $\hat{\psi}_m$ we know that
\begin{equation}\label{eq:inegalite_gamma_et_hat_m}
\gamma_n(\hat{\psi}_{\hat{m}}) + \text{pen}(\hat{m}) \leq \gamma_n(\hat{\psi}_{m}) + \text{pen}(m) \leq \gamma_n(\psi_{m}) + \text{pen}(m).
\end{equation}
There is a relation between the process $\nu_n$ defined in Lemma~\ref{Lemme_de_Talagrand} and the contrast $\gamma_n$. Indeed using~\eqref{eq:Esperance_et_psi}, if $t$ and $s$ are two function then
\begin{align*}
\gamma_n(t) - \gamma_n(s) &= \| t\|_2^2 - \|s\|_2^2 -\frac{2}{n} \sum_{i=1}^n \left(\Delta_i t(X_i) - \Delta_i s(X_i) \right)\\
&= \|t\|^2_2 - 2<t,\psi> + \|\psi\|^2_2 \\
& -\|s\|^2_2 + 2<s,\psi> - \|\psi\|^2_2 \\
& - \frac{2}{n} \sum_{i=1}^n \left( \Delta_i (t-s)(X_i) - <t-s, \psi> \right)\\
&= \|t-\psi\|_2^2 - \|s-\psi\|_2^2 -\frac{2}{n}\sum_{i=1}^n (l_{t-s}(X_i,\Delta_i) - \Esp(l_{t-s}(X_i,\Delta_i)))\\
&= \|t-\psi\|_2^2 - \|s-\psi\|_2^2 -2\nu_n(l_{t-s}).
\end{align*}
Using this decomposition on inequality \eqref{eq:inegalite_gamma_et_hat_m} we obtain
\[ \text{pen}(\hat{m}) - \text{pen}(m) \leq \gamma_n(\psi_m) - \gamma_n(\hat{\psi}_{\hat{m}}) = \|\psi_m - \psi \|_2^2 - \|\hat{\psi}_{\hat{m}} - \psi\|_2^2 - 2 \nu_n(l_{\psi_m - \hat{\psi}_{\hat{m}}}). \]
Using this, and the fact that $t\mapsto \nu_n(t)$ is linear, we have the following inequality:
\begin{gather}
\nonumber\|\hat{\psi}_{\hat{m}} - \psi\|_2^2  \leq \text{pen}(m) - \text{pen}(\hat{m}) + \| \psi_m - \psi\|_2^2 + 2\nu_n\left(l_{\hat{\psi}_{\hat{m}} - \psi_m}\right) \\
\nonumber\leq \text{pen}(m) - \text{pen}(\hat{m}) + \| \psi_m - \psi\|_2^2 + 2\|\hat{\psi}_{\hat{m}} - \psi_m \|_2 \left(\sup_{t\in B_{m,\hat{m}}(0,1)}\nu_n(l_t)\right)\\
\label{eq:intermediaire} \leq \text{pen}(m) - \text{pen}(\hat{m}) + \| \psi_m - \psi\|_2^2 + x^{-1}\|\hat{\psi}_{\hat{m}} - \psi_m \|_2^2 + x\left(\sup_{t\in B_{m,\hat{m}}(0,1)}\nu_n^2(l_t)\right),
\end{gather}
where $B_{m,\hat{m}}(0,1) = \left\{t\in S_m + S_{\hat{m}}| \|t\|_2 \leq 1  \right\}$,  and we use that $\forall (a,b)\in\mathbb{R}_+^2, \forall x>0, 2ab \leq x^{-1}a^2 + xb^2$. We know that for every $y$ positive, $$\| \hat{\psi}_{\hat{m}} -\psi_m \|_2^2 \leq (1 + y^{-1}) \| \hat{\psi}_{\hat{m}} - \psi \|_2^2 + (1+y) \| \psi_m - \psi\|_2^2.$$ We take $y=C_x := \frac{x+1}{x-1}$ with $x>1$ such that $C_x >0$ and $1- x^{-1}\left(1+C_x^{-1}\right) = C_x^{-1}$ and $1+ x^{-1}(1+C_x) = C_x$. Thus using this inequality on \eqref{eq:intermediaire} we obtain
\begin{equation}\label{eq:intermediaire2}
C_x^{-1}\|\hat{\psi}_{\hat{m}} - \psi\|_2^2 \leq \text{pen}(m) - \text{pen}(\hat{m})+ C_x\|\psi_{m} - \psi \|_2^2 + x\sup_{t\in B_{m,\hat{m}}(0,1)}\nu_n^2(l_t) .\\
\end{equation}
We set a function $\text{p}(\cdot,\cdot)$ such that, for all $m,m'$ in $\mathcal{M}_n$,
\begin{equation}\label{eq:Inegalité_pen_p}
 x \text{p}(m,m') \leq \text{pen}(m) + \text{pen}(m').
\end{equation}
With this in mind we come back to \eqref{eq:intermediaire2}
\begin{align}
\nonumber C_x^{-1}\|\hat{\psi}_{\hat{m}} - \psi\|_2^2 &\leq \text{pen}(m) - \text{pen}(\hat{m})+ C_x\|\psi_{m} - \psi \|_2^2 + x\left(\sup_{t\in B_{m,\hat{m}}(0,1)}\nu_n^2(l_t)\right) \\
\nonumber &\leq \text{pen}(m) - \text{pen}(\hat{m})+ C_x\|\psi_{m} - \psi \|_2^2 \\
\nonumber &+ x \text{p}(m,\hat{m})+ x\left(\sup_{t\in B_{m,\hat{m}}(0,1)}\nu_n^2(l_t) - \text{p}(m,\hat{m}) \right) \\
\nonumber &\leq 2\text{pen}(m) + C_x\|\psi_{m} - \psi \|_2^2+ x\left(\sup_{t\in B_{m,\hat{m}}(0,1)}\nu_n^2(l_t) - \text{p}(m,\hat{m}) \right).
\end{align}
So in the end we obtain
\begin{equation}\label{eq:adpatation_quasi_finie}
\|\hat{\psi}_{\hat{m}} - \psi\|_2^2 \leq C_x^2 \|\psi_{m} - \psi \|_2^2+ 2C_x\text{pen}(m)+ xC_x\left(\sup_{t\in B_{m,\hat{m}}(0,1)}\nu_n^2(l_t) - \text{p}(m,\hat{m}) \right) .
\end{equation}
To end the proof we have to show that
\begin{align}
\nonumber\Esp\left(\sup_{t\in B_{m,\hat{m}}(0,1)}\nu_n^2(l_t) - \text{p}(m,\hat{m})\right)_+ &\leq \sum_{m'\in\mathcal{M}_n} \Esp\left( \sup_{t\in B_{m,m'}(0,1)}\nu_n^2(l_t) - \text{p}(m,m') \right)_+ \\
\label{eq:besoin_du_lemme}&\leq \frac{C_1}{n}.
\end{align}
To prove this upper bound we use Lemma~\ref{Lemme_de_Talagrand}.
\begin{eqnarray*}
&&\sum_{m'\in\mathcal{M}_n} \Esp\left( \sup_{t\in B_{m,m'}(0,1)}\nu_n^2(l_t) - \text{p}(m,m') \right)_+ \\
&\leq& \sum_{m'\in\mathcal{M}_n}\frac{\kappa_1}{n} \left( e^{-\kappa_2\epsilon (D_m + D_{m'})} + \frac{e^{-\kappa_3 C(\epsilon)\sqrt{\epsilon}\sqrt{n}}}{C(\epsilon)^2} \right) \\
&\leq& \frac{\kappa_1}{n} \left( \sum_{m'\in\mathcal{M}_n}e^{-\kappa_2\epsilon (D_m + D_{m'})} \right) + \frac{\kappa_1 |\mathcal{M}_n|e^{-\kappa_3 C(\epsilon)\sqrt{\epsilon}\sqrt{n}}}{nC(\epsilon)^2} \\
&\leq& \left(\frac{\kappa_1 e^{-\kappa_2\epsilon D_m}}{n} \underbrace{\sum_{k=0}^n e^{-\kappa_2 \epsilon k}}_{< + \infty}\right) + \frac{\kappa_1 e^{-\kappa_3 C(\epsilon)\sqrt{\epsilon}\sqrt{n}}}{C(\epsilon)^2} \leq \frac{C_1}{n},
\end{eqnarray*}
for a constant $C_1>0$. Here we choose $\epsilon=\frac{1}{2}$ and recall that $|\mathcal{M}_n|\leq n$. 
Moreover we have that $\text{p}(m,\hat{m}) =4\Phi_0^2\frac{(D_m + D_{\hat{m}})}{n}\Esp(\Delta) $. So if we take $\kappa  = 4x >4$ we have
\[ \text{pen}(m) \geq \kappa\Phi_0^2 \frac{D_m}{n}\Esp(\Delta), \]
we will have inequality~\eqref{eq:Inegalité_pen_p} verified.\\
So, if we write $C_\kappa = \frac{\kappa + 4}{\kappa - 4}$, we obtain that for all $m\in\mathcal{M}_n$,
\begin{align*}
\Esp\left(\|\hat{\psi}_{\hat{m}} - \psi\|_2^2\right) &\leq C_{\kappa}^2 \|\psi_{m} - \psi \|_2^2+ 2C_{\kappa}\text{pen}(m)+ \frac{\kappa C_1 C_{\kappa}}{4n}\\
&\leq \breve{C} \left(  \|\psi_{m} - \psi \|_2^2 + \text{pen}(m) \right) + \frac{\mathring{C}}{n},
\end{align*}
where $\breve{C}$ depends on $\kappa$ and $\mathring{C}$ depends on $\kappa$, $\Phi_0$ and $\|\psi\|_\infty$. This inequality holds for all $m\in \mathcal{M}_n$.\\
Thus we have the wanted inequality 
\[ \Esp\left( \| \hat{\psi}_{\hat{m}} - \psi \|_2^2 \right) \leq  \breve{C}\inf_{m\in \mathcal{M}_n} \left(\|\psi - \psi_m\|_2^2 + \text{pen}(m) \right) + \frac{\mathring{C}}{n}.\]

\subsection{Proof of Theorem~\ref{Penalisation_aleatoire}} \label{Proof Penalisation_aleatoire}
To prove this theorem we use Proposition~\ref{Penalisation_première} and some ideas of its proof (see Section~\ref{Proof Penalisation_première}).\\
We set the following random set, for $b\in ]0, 1[$,
\[ S_b = \left\{\omega\in\Omega, \left| \frac{\frac{1}{n} \sum_{i=1}^n \Delta_i}{ \Esp(\Delta)} - 1 \right| < b\right\} .\]
On $S_b$ the following inequalities are true
\[ \frac{1}{n}\sum_{i=1}^n \Delta_i  < (b+1)\Esp(\Delta), \hspace{10pt}\Esp(\Delta)< \frac{1}{(1-b)n}\sum_{i=1}^n \Delta_i.  \]
With those inequalities we can follow the arguments and computations of the proof of Proposition~\ref{Penalisation_première}.
So with a penalization term defined as $ \widehat{\text{pen}}(m) = 2x(1+2\epsilon)\Phi_0^2\frac{D_m}{n} \frac{1}{(1-b)n}\sum_{i=1}^n \Delta_i $, for $\epsilon>0$, $x>1$ and $C_x = \frac{x+1}{x-1}$, using \eqref{eq:adpatation_quasi_finie} and \eqref{eq:besoin_du_lemme}, we would obtain the following, for all $m\in\mathcal{M}_n$:
\[ \Esp\left(\|\hat{\psi}_{\hat{m}} - \psi\|_2^2 \mathds{1}_{S_b}\right) \leq C_x^2 \|\psi_m - \psi\|_2^2 +\frac{2C_x \Phi_0^2}{(1-b)}\frac{D_m}{n}\Esp(\Delta) + \frac{Q}{n}.\]
where $Q$ depends on $\kappa, \Phi_0$ and $\|\psi\|_\infty$.\\
So all there is to do to obtain our inequality is to prove that
\[ \Esp\left(\|\hat{\psi}_{\hat{m}} - \psi\|_2^2\mathds{1}_{S_b^c}\right) \leq \frac{Q'}{n}. \]
Using \eqref{eq:Ineg9.9} we know that
\[ \|\hat{\psi}_{\hat{m}} - \psi\|_2^2 \leq 2\left( \|\hat{\psi}_{\hat{m}}\|_2^2 +  \|\psi\|_2^2\right) \leq 2\left(\Phi_0^2 D_{\hat{m}} + \|\psi\|_2^2\right) \leq 2\left(\Phi_0^2 n + \|\psi\|_2^2\right).\]
Thus we have
\[\Esp\left( \|\hat{\psi}_{\hat{m}}- \psi\|_2^2\mathds{1}_{S_b^c}\right) \leq 2\left(\Phi_0^2 n + \|\psi\|_2^2\right)\Proba(S_b^c).\]
It remains to show that $\Proba(S_b^c) = \mathcal{O}\left(n^{-2}\right)$. For that we use Hoeffding's inequality with $B= b\Esp(\Delta)$,
\begin{align*}
\Proba(S_b^c) &= \Proba\left( \left| \frac{1}{n}\sum_{i=1}^n \Delta_i - \Esp(\Delta) \right| \geq B \right) \leq \Proba\left(\left| \sum_{i=1}^n \Delta_i - n\Esp(\Delta) \right| \geq nB\right)\\
&\leq 2\exp\left(- \frac{2(nB)^2}{\sum_{i=1}^n 1^2}\right) = 2\exp\left(-2nB^2\right) \leq 2\exp\left(-2\sigma_0^2 b^2 n\right) = \mathcal{O}\left(n^{-2}\right).
\end{align*}
So we obtain the announced upper bound
\[ \Esp\left( \|\hat{\psi}_{\hat{m}}- \psi\|_2^2\mathds{1}_{S_b^c}\right) \leq \frac{Q'}{n},\]
where $Q'$ depends on $\Phi_0,\sigma_0$ and $\|\psi\|_\infty$.\\
In the end we take $b=\epsilon = \frac{1}{2}$ thus having $\widehat{\text{pen}}(m) = 8x\Phi_0^2\frac{D_m}{n} \frac{1}{n}\sum_{i=1}^n \Delta_i $. So we take $\kappa = 8x > 8$ and with $C'_\kappa = \frac{\kappa + 8}{\kappa - 8}$ we obtain
\begin{align*}
 \Esp\left(\|\hat{\psi}_{\hat{m}}- \psi\|_2^2\right) &\leq  {C'_{\kappa}}^2 \|\psi_m - \psi\|_2^2 + 2C'_\kappa \Phi_0^2\frac{D_m}{n}\Esp(\Delta) + \frac{Q}{n} + \frac{Q'}{n}\\
&\leq C\left( \|\psi - \psi_m\|_2^2 + \Phi_0^2 \frac{D_m}{n}\Esp(\Delta) \right) + \frac{\tilde{C}}{n},
\end{align*}
where $C$ depends only on $\kappa$, and $\tilde{C}$ depends on $\kappa, \Phi_0, \sigma_0$ and $\|\psi\|_\infty$. This inequality holds for all $m\in \mathcal{M}_n$.\\
Thus we obtain the wanted inequality
\[ \Esp\left( \| \hat{\psi}_{\hat{m}} - \psi \|_2^2 \right) \leq  C\inf_{m\in \mathcal{M}_n} \left(\|\psi - \psi_m\|_2^2 + \Phi_0^2 \frac{D_m}{n}\Esp(\Delta) \right) + \frac{\tilde{C}}{n}.\]

\subsection{Proof of Proposition~\ref{Adaptation seuil}}\label{Proof Adaptation seuil}

We recall from Section~\ref{Preuve MISE estimateur seuil} we have
\begin{align*}
\Esp\left(\|\hat{f}^* - f\|_2^2\right) &=\Esp\left(\|\hat{f}^* - f\|_2^2\mathds{1}_{E}\right) + \Esp\left(\|\hat{f}^* - f\|_2^2\mathds{1}_{E^c}\right)\\
&\leq \frac{8}{\sigma_0^2}\left(\Esp\left(\|\hat{\psi}_{\hat{m}} - \psi\|_2^2\right) + \Esp\left(\|(\sigma - \hat{\sigma}) f\|_2^2\right) \right)  + \Esp\left(\|\hat{f}^* - f\|_2^2\mathds{1}_{E^c}\right)\\
&\leq \frac{8}{\sigma_0^2}\left(\Esp\left(\|\hat{\psi}_{\hat{m}} - \psi\|_2^2\right) + \frac{\|f\|_2^2}{4n}\right) +  \Esp\left(\|\hat{f}^* - f\|_2^2\mathds{1}_{E^c}\right).
\end{align*}
For the last term we use that on $E^c$ we have $\hat{\sigma}\lor n^{-1/2} \geq n^{-1/2}$ thus
\begin{align*}
\|\hat{f}^* - f\|_2^2 &\leq 2\left\|\frac{(\hat{\psi}_{\hat{m}})_+}{\hat{\sigma}\lor n^{-1/2}}\right\|_2^2 + 2\|f\|_2^2\leq 2n\|(\hat{\psi}_{\hat{m}})_+\|_2^2 + 2\|f\|_2^2\\
&\leq 2n\|\hat{\psi}_{\hat{m}}\|_2^2 + 2\|f\|_2^2 \leq 2n^2\Phi_0^2 + 2\|f\|_2^2 ,\\
\end{align*}
using \eqref{eq:Ineg9.9}.\\
And because $E^c\subset \left\{ \omega\in\Omega,\|\hat{\sigma} - \sigma\|_{\infty} \geq \frac{\sigma_0}{2} \right\}$ thanks to Lemma~\ref{Lemme} we have
\begin{align*}
\Esp\left(\|\hat{f}^* - f\|_2^2\mathds{1}_{E^c}\right) \leq \left(2n^2\Phi_0^2 + 2\|f\|_2^2\right)\Proba(E^c) &\leq \left(n^2\Phi_0^2 + 2\pi\|f\|_{\infty}^2\right)12e^{-2n\frac{\sigma_0^2}{36}}\\
&\leq \frac{C_2}{n}.
\end{align*}
Thus using Theorem~\ref{Penalisation_aleatoire} we obtain
\begin{align*}
 \Esp\left( \|\hat{f}^* - f \|_2^2 \right) &\leq  C\inf_{m\in \mathcal{M}_n} \left(\|\psi - \psi_m\|_2^2 +\Phi_0^2 \frac{D_m}{n}\Esp(\Delta) \right) + \left( \tilde{C} + \frac{\pi\|f\|_{\infty}^2}{2} + C_2\right) \frac{1}{n}\\
&\leq  K \inf_{m\in \mathcal{M}_n} \left(\|\psi - \psi_m\|_2^2 +\Phi_0^2 \frac{D_m}{n}\Esp(\Delta) \right) + \frac{\tilde{K}}{n},
\end{align*}
where $K:=C$ depends only on $\kappa$ and $\tilde{K}:= \tilde{C} + \frac{\pi\|f\|_{\infty}^2}{2} + C_2$ is a constant that depends on $\kappa, \Phi_0, \sigma_0$ and $\|f\|_\infty$. We then obtain the oracle inequality of the proposition.

\definecolor{codegreen}{rgb}{0,0.6,0}
\definecolor{codegray}{rgb}{0.5,0.5,0.5}
\definecolor{codepurple}{rgb}{0.58,0,0.82}
\definecolor{backcolour}{rgb}{0.95,0.95,0.92}

\section*{Bibliography}

\printbibliography[heading=none]

\end{document}